\documentclass[12pt,reqno]{amsart} 
\usepackage{amsmath}
\allowdisplaybreaks[4]
\usepackage{amsfonts}
\usepackage{amssymb}
\usepackage{amsthm}
\usepackage{amstext}
\usepackage{amsmath,amscd}
\usepackage{latexsym}
\usepackage{geometry}
\usepackage{fancyhdr}
\usepackage[numbers,sort&compress]{natbib}
\usepackage{hypernat}
\allowdisplaybreaks
\numberwithin{equation}{section}
\usepackage{array}
\usepackage{graphicx} 
\usepackage{hyperref} 
\hypersetup{
	colorlinks=true,
	linkcolor=blue, 
	urlcolor=red,
	citecolor=green,
}

\author{Wenju Wu, Fulin Zhong$^\dag$}
\title{Existence of periodic solutions for the Grushin critical problem}

\address{School of Mathematics and Statistics, Central China Normal University, Wuhan 430079, P. R. China}
\email{wjwu@mails.ccnu.edu.cn} 

\address{School of Mathematics and Statistics, Central China Normal University, Wuhan 430079, P. R. China}
\email{flzhong@mails.ccnu.edu.cn}

\thanks{The research was supported by NSFC (No.~12471106) and the self-determined research funds of CCNU from the colleges' basic research and operation of MOE (No.~CCNU24JC006).}
\thanks{$\dag$Corresponding author: Fulin Zhong.}

\usepackage{fullpage}
\numberwithin{equation}{section}
\newtheorem{theorem}{Theorem}[section]
\newtheorem{remark}[theorem]{Remark}
\newtheorem{lemma}[theorem]{Lemma}
\newtheorem{proposition}[theorem]{Proposition}
\newtheorem{corollary}[theorem]{Corollary}

\begin{document}	
	\maketitle
	\pagestyle{fancy}
	\lhead{}
	\chead{}
	\rhead{}
	\lfoot{}
	\cfoot{\thepage}
	\rfoot{}
	\renewcommand{\headrulewidth}{0pt} 
	\renewcommand{\footrulewidth}{0pt} 

	\begin{abstract}
		We study a Grushin critical problem in a strip domain which satisfies the periodic boundary conditions. By applying the finite-dimensional reduction method, we construct a periodic solution when the prescribed curvature function is periodic. Furthermore, we also consider the Grushin critical problem in $\mathbb{R}^{N} (N \geq 5)$. Compared with Billel et al. (Differential Integral Equations 32: 49-90, 2019), we use the method by Guo and Yan (Math. Ann. 388: 795-830, 2024) to construct periodic solutions under some weaker conditions, avoiding the complicated estimates and uniqueness proof. Notably, Guo and Yan (Math. Ann. 388: 795-830, 2024) obtained solutions periodic with respect to some of the first variables, while the solutions in this paper are periodic with respect to some intermediate variables. 
		
		
		\medskip\noindent
		{\bf AMS:} 35B10; 35B33; 35B44.
	\end{abstract}

	\section{Introduction} 
	In this paper, we study the following Grushin critical problem with periodic boundary conditions
	\begin{align}\begin{aligned}\label{eq1.1}
			\begin{cases}
				-\Delta u(x)=M(x) \frac{u^{\frac{N}{N-2}}(x)}{|y|}, \ u>0, & \text { in } \Omega, \\
				u\left(y, z^1+L e_{j},z^2\right)=u(x), u_{x_{j}}\left(y, z^1+L e_{j},z^2\right)=u_{x_{j}}(x), & \text { if } z^1=-\frac{L}{2} e_{j}, j=1, \cdots, \bar{k}, \\
				u \rightarrow 0, & \text { as } \left|y\right|+\left|z^2\right| \rightarrow+\infty,
			\end{cases}
		\end{aligned}
	\end{align}
	where $N \ge 5$, $L>0$ is a constant, $\Omega$ is a strip defined as
	$$
	\Omega=\Bigl\{x=(y,z)\in \mathbb{R}^{N}: z=(z^1,z^2),\,|z^1_i| \le \frac{L}{2},\,i=1,\cdots,\bar{k},\,y \in \mathbb{R}^{k},\,z\in \mathbb{R}^{h}\Bigr\},
	$$
	$1\le h\le k-1$, $k+h=N$, $1\le \bar{k}\le h$ and $e_1=(1,0,\cdots,0)$, $e_2=(0,1,\cdots,0), \cdots, e_{h}=(0,0,\cdots,1)$.
	
	Our primary motivation to study \eqref{eq1.1} is to construct solutions that are periodic in their $\bar{k}$ variables for the following prescribed scalar curvature problem in $\mathbb{R}^{N}$, if the curvature function $M(x)$ is periodic in its the $\bar{k}$ variables $(z_1,\cdots,z_{\bar{k}})$:
	\begin{equation}\label{eq1.2}
		-\Delta u(x)=M(x) \frac{u^{\frac{N}{N-2}}(x)}{|y|}, \ u>0, \ x=(y,z)\in \mathbb{R}^{k} \times \mathbb{R}^{h}=\mathbb{R}^{N}.
	\end{equation}
	
	Problem \eqref{eq1.2} is a special form of the following problem related to the Grushin operator
	\begin{equation}\label{beq1.3}
		\mathbb{G} u:=-\Delta_{y} u-4|y|^2 \Delta_{z} u=M(x) u(x)^{\frac{Q+2}{Q-2}}, \ u>0, \ x=(y, z) \in \mathbb{R}^{m_1} \times \mathbb{R}^{m_2},
	\end{equation}
	where $\{m_1, m_2\} \subset \mathbb{N}^+$, $Q=m_1+2 m_2$ is the appropriate dimension and $\frac{Q+2}{Q-2}$ is the corresponding critical exponent. In fact, if $M(x)=M(|y|, z)$ and $u=\psi(|y|, z)$ satisfy \eqref{beq1.3}, then
	\begin{equation*}\label{beq1.5}
		-\psi_{r r}(r, z)-\frac{m_1-1}{r} \psi_{r}(r, z)-4 r^2 \Delta_{z} \psi(r, z)=M(|y|, z) \psi(r, z)^{\frac{Q+2}{Q-2}},
	\end{equation*}
	where $r=|y|$. Defining $v(r, z)=\psi(\sqrt{r}, z)$, then
	$$
	\psi_{r}(\sqrt{r}, z)=2 \sqrt{r} v_{r}(r, z), \ \psi_{r r}(\sqrt{r}, z)=4 r v_{r r}(r, z)+2 v_{r}(r, z).
	$$
	Hence, $v$ satisfies
	\begin{equation*}
		-v_{r r}(r, z)-\frac{m_1}{2 r} v_{r}(r, z)-\Delta_{z} \psi(r, z)=\frac{M(\sqrt{r}, z)}{4 r} v(r, z)^{\frac{Q+2}{Q-2}},
	\end{equation*}
	that is, $v=v(|y|, z)$ solves the Hardy-Sobolev-type problem
	\begin{equation}\label{bbeq1.7}
		-\Delta u(x)=M(x) \frac{u^{\frac{k+h}{k+h-2}}}{|y|}, \ x=(y, z) \in \mathbb{R}^{k} \times \mathbb{R}^{h},
	\end{equation}
	where $k=\frac{m_1+2}{2}$, $h=m_2$ and $M(x)=M(|y|, z)=\frac{M(\sqrt{r}, z)}{4}$.	
	If we denote $k+h=N$, then \eqref{bbeq1.7} becomes \eqref{eq1.2}.
	
	The general Grushin semilinear critical problem is
	\begin{equation}\label{beq1.4}
		-\Delta_{y} u-(\alpha+1)^2|y|^{2 \alpha} \Delta_{z} u=M(x) u(x)^{\frac{Q_{\alpha}+2}{Q_{\alpha}-2}}, \ u>0, \ x=(y, z) \in \mathbb{R}^{m_1} \times \mathbb{R}^{m_2},
	\end{equation}
	where $\alpha \geq 0$ is a real number, $\mathbb{G}_{\alpha}:=-\Delta_{y} -(\alpha+1)^2|y|^{2 \alpha} \Delta_{z}$ is known as the Grushin operator, $Q_{\alpha}=m_1+m_2(\alpha+1)$ is the appropriate homogeneous dimension and the power $\frac{Q_{\alpha}+2}{Q_{\alpha}-2}$ is the corresponding critical exponent.
	
	If $\alpha=0$ and $m_1+m_2=N$, then \eqref{beq1.4} reduces to 
	\begin{equation}\label{beq1.6}
		-\Delta u(x) = M(x)u(x)^{\frac{N+2}{N-2}}, \quad u>0 \quad \text { in } \mathbb{R}^{N}.
	\end{equation}
	Via the stereographic projection, \eqref{beq1.6} is equivalent to the prescribing scalar curvature problem on the standard $n$-sphere $\mathbb{S}^{N}$, that is
	\begin{equation}\label{beq1.7}
		-\Delta_{\mathbb{S}^{N}} u+\frac{N(N-2)}{4} u=\frac{N-2}{4(N-1)} M(x) u^{\frac{N+2}{N-2}}, \quad u>0 \quad \text { on } \mathbb{S}^{N}.
	\end{equation}
	The existence of solutions for these problems \eqref{beq1.6} and \eqref{beq1.7} has been investigated widely in the last decades. We refer the readers to \cite{DLY,LYY1,LYY2,LWX,NWM,PWW,WY} and the references therein. For the fractional case of \eqref{beq1.6}, we can see \cite{GNNT,NTW}. 
	
	In the case $\alpha=1$, the nonlinear equation \eqref{beq1.3} already appeared in connection with the Cauchy-Riemann Yamabe problem, which was solved by Jerison and Lee in \cite{JL1,JL2}. For more general choices of $m_1$ and $m_2$ or for general $\alpha>0$, positive entire solutions of \eqref{beq1.4} with $M(x) \equiv 1$ have been discussed in \cite{GV,MM}. In the case of $M(x) \not \equiv 1$, the existence of multiple positive solutions has been investigated in \cite{CFMS,CPY,FU,MU}.
	
	When $M = M(|y|,z)$ is a cylindrical function, \eqref{eq1.2} has been studied extensively. Using variational methods, Cao et al. \cite{CPY} constructed multi-peak solutions for \eqref{eq1.2} which concentrate exactly at two points with a potentially large distance between them. Wang et al. \cite{WWY} proved that \eqref{eq1.2} has infinitely many positive solutions with cylindrical symmetry by a Lyapunov-Schmidt reduction argument, whose energy can be made arbitrarily large. Under some suitable conditions on $M(x)$ near its critical point,  Gheraibia et al. \cite{BWY} proved that the \eqref{eq1.2} has infinitely many bubbling solutions, which are locally unique and partially periodic. We refer the readers to \cite{GGZ,LN,LTW,LW} for other results of the existence of solutions for \eqref{eq1.2}. Moreover, Guo and Yan \cite{Guo-Yan} proved that an elliptic problem involving critical growth has infinitely many solutions, which are $L$-periodic in its first $k$ variables for any sufficiently large integer $L > 0$. It is very novel and subtle that they constructed solutions which are periodic in some variables under the weakest possible
	conditions, without involving the complicated estimates in \cite{LWX} and the uniqueness proof in \cite{DLY}. Also, there have been some results on the existence of periodic solutions for other equations and systems of equations, see \cite{Guo-Wu-Yan,Guo-Wu1,Guo-Wu2,WWZ}. Motivated by \cite{BWY,Guo-Yan}, we intend to study the existence of the periodic solutions for \eqref{eq1.2}.  
	
	In this paper, we assume that the function $M(x)$ satisfies the following conditions:
	
	$\left(A_1\right)$ $M$ is 1-periodic in its $\bar{k}$ variables $(z_1,\cdots,z_{\bar{k}})$;
	
	$\left(A_2\right)$ $0<\inf _{x \in \mathbb{R}^N} M(x) \leq \sup _{x \in \mathbb{R}^N} M(x)<\infty$ and there exist constants $a_{i} \neq 0$, $\beta_{i} \in(N-2, N-1)$, $i=1, \cdots, N$, such that $\beta:=\min _{1 \leq i \leq N} \beta_{i} \leq \beta_{M}:=\max _{1 \leq i \leq N} \beta_{i} \leq \beta\left(1+\frac{1}{N-2}\right)$, and
	
	$$
	M(x)=1+\sum_{i=1}^{N} a_{i}\left|x_{i}\right|^{\beta_{i}}+O\left(|x|^{\beta_{M}+\kappa}\right), \ x \in B_{\delta}(0),
	$$	
	where $\delta>0$ and $\kappa>0$ are some sufficiently small constants. Furthermore, we assume that $\sum_{j \in J} a_{j}<0$, where $J=\left\{j: \beta_{j}=\beta\right\}$ satisfies $J \subset \{1,2,\cdots,k\}$ or $J \subset \{k+1,k+2,\cdots,N\}$.
	
	Herein, we present the main results of the paper.
	\begin{theorem}\label{th1.1}
		Suppose that $M(x)$ satisfies the conditions $\left(A_1\right)$ and $\left(A_2\right)$. If $N=k+h \ge 5$, $1\le h\le k-1$ and $1 \leq \bar{k}<\frac{N-2}{2}$, then \eqref{eq1.1} has a solution $u_{L}$, if the integer $L>0$ is sufficiently large.
	\end{theorem}
	
	As a direct consequence of Theorem \ref{th1.1}, we derive the following results.
	\begin{corollary}\label{cor1.2}
		Under the same conditions as in Theorem \text{\rm \ref{th1.1}}, \eqref{eq1.2} has infinitely many solutions, which are $L$-periodic in its $\bar{k}$ variables for any sufficiently large integer $L>0$.
	\end{corollary}
	
	\begin{remark} \label{re1.3}
		The condition $J \subset \{1,2,\cdots,k\}$ or $J \subset \{k+1,k+2,\cdots,N\}$ is essential in the proof of Proposition \text{\rm \ref{proA.4}}. Moreover, the condition $1\le h\le k-1$ is equivalent to $\frac{N+1}{2}\le k\le N-1$, which is crucial for using Lemma B.2 in \text{\rm \cite{WWY}} in our paper.
	\end{remark}
	
	\begin{remark} \label{re1.4}
		Compared to \text{\rm \cite{BWY}}, we adopt the technique in \text{\rm \cite{Guo-Yan}} to construct periodic solutions for \eqref{eq1.2} under weaker conditions, avoiding the complicated estimates and uniqueness proof. 
	\end{remark}
	
	\begin{remark} \label{re1.5}
		The solutions constructed in Corollary \text{\rm \ref{cor1.2}} are periodic with respect to some intermediate variables, which is a distinction from those in \text{\rm \cite{Guo-Yan}}.
	\end{remark}
	
	Throughout this paper, we define $2^{\sharp}=\frac{2(N-1)}{N-2}$,
	$$
	\mathcal{D}^{1,2}\left(\mathbb{R}^{N}\right)=\left\{\int_{\mathbb{R}^{N}} \frac{|u|^{2^{\sharp}}}{|y|} \mathrm{d} x<+\infty:|\nabla u| \in L^{2}\left(\mathbb{R}^{N}\right)\right\}
	$$
	and $\mathcal{D}^{1,2}\left(\mathbb{R}^{N}\right)$ endows the norm $\|u\|=\left(\int_{\mathbb{R}^{N}} |\nabla u|^{2} \mathrm{d} x\right)^{\frac{1}{2}}$, which is induced by the inner product $\langle u, v\rangle=\int_{\mathbb{R}^{N}} \nabla u \cdot \nabla v \mathrm{d} x$.
	It is well known that (see \cite{FMS}) all the positive solutions to the following problem
	$$
	-\Delta u(x)=\frac{u^{\frac{N}{N-2}}(x)}{|y|}, \ u>0, \ x=(y,z)\in \mathbb{R}^{k} \times \mathbb{R}^{h}=\mathbb{R}^{N}, \ u \in \mathcal{D}^{1,2}(\mathbb{R}^N)
	$$
	are given by
	$$
	W_{\hat{x},\lambda}(y, z)=\frac{C_{N,k}\lambda^{\frac{N-2}{2}}}{\left((1+\lambda |y|)^2+\lambda^2|z-\hat{z}|^2\right)^{\frac{N-2}{2}}}, \ \hat{z} \in \mathbb{R}^{h}, \ \lambda>0,
	$$
	where $C_{N,k}=[(N-2)(k-1)]^{\frac{N-2}{2}}$ and $\hat{x}=(0,\hat{z})\in \mathbb{R}^{k} \times \mathbb{R}^{h}$.
	
	To prove Theorem \ref{th1.1}, we aim to construct a solution $u_{L}$ for \eqref{eq1.1}, which satisfies $u_{L} \approx W_{\hat{x},\lambda}$ for some $\hat{x}$ close to $0$ and $\lambda>0$ sufficiently large. To this end, taking the boundary conditions in \eqref{eq1.1} into consideration, we first construct the approximate solution $PW_{\hat{x},\lambda}$, which solves the following problem 
	\begin{align}\begin{aligned}\label{eq1.3}
			\begin{cases}
				-\Delta u(x)=\frac{{W_{\hat{x},\lambda}}^{\frac{N}{N-2}}(x)}{|y|}, & \text { in } \Omega, \\
				u\left(y, z^1+L e_{j},z^2\right)=u(x), u_{x_{j}}\left(y, z^1+L e_{j},z^2\right)=u_{x_{j}}(x), & \text { if } z^1=-\frac{L}{2} e_{j}, j=1, \cdots, \bar{k}, \\
				u \rightarrow 0, & \text { as }\left|y\right|+\left|z^2\right| \rightarrow+\infty.
			\end{cases}
		\end{aligned}
	\end{align}

	For any integer $\bar{k} \in[1, h]$, we define $\bar{k}$-dimensional lattice by
	$$
	Q_{\bar{k}}:=\left\{\left(y_1, \cdots, y_{\bar{k}}, 0\right)\in \mathbb{R}^{h}: y_{i} \text { is an integer, } i=1, \cdots, \bar{k};\, 0 \in \mathbb{R}^{h-\bar{k}}\right\}.
	$$
	We order all the points in $Q_{\bar{k}}$ as $\tilde{\left\{P^{j}\right\}}_{j=1}^{\infty}$ and $\tilde{P}^{0}=0$. Denote $P^{j}=(0,\tilde{P}^{j})\in \mathbb{R}^{N}$ and $P_{L}^{j}=(0,L\tilde{P}^{j})\in \mathbb{R}^{N}$, where $0 \in \mathbb{R}^{k}$.
	The Green's function $G(x,\xi)$ of $-\Delta$ in $\Omega$ satisfying the boundary conditions in \eqref{eq1.1} is given by
	$$
	G(x,\xi)=\sum_{j=0}^{\infty} \Gamma\left(x, \xi+P_{L}^{j}\right), \ x=(y,z),\xi=(\xi^1,\xi^2) \in \Omega,
	$$
	where
	$$
	\Gamma(x,\xi)=\frac{1}{(N-2)\omega_{N-1}}\frac{1}{|x-\xi|^{N-2}}
	$$
	and $\omega_{N-1}$ is the area of $\mathbb{S}^{N-1}$. Then by the Green representation formula, we have that
	\begin{equation*}\label{eq1.4}
		PW_{\hat{x},\lambda}(x)=\int_\Omega G(x,\xi)\frac{{W_{\hat{x},\lambda}}^{\frac{N}{N-2}}(\xi)}{|\xi^1|}\mathrm{d}\xi
	\end{equation*}
	solves \eqref{eq1.3}. Let us also point out that the bounded solution of \eqref{eq1.3} is unique. In fact, for any bounded solutions $u_1$ and $u_2$ of \eqref{eq1.3}, letting $v=u_1-u_2$, then $-\Delta v=0$ in $\Omega$, which can be extended periodically to $\mathbb{R}^N$. Thus $v$ must be a constant, which, together with $v \to 0$ as $\left|y\right|+\left|z^2\right| \rightarrow+\infty$, gives $v=0$.
	
	In order to construct a solution of the form $u_{L}=PW_{\hat{x}_{0,L}, \lambda_{L}}+\omega_{L}$, where $\hat{x}_{0,L}$ is close to $0$ and $\lambda_{L}>0$ is sufficiently large, we choose $\hat{x}_L^j=(0,\hat{z}_L^j)\in \mathbb{R}^{k} \times \mathbb{R}^{h}$, $\hat{x}=(0,\hat{z})\in \mathbb{R}^{k} \times \mathbb{R}^{h}$ and introduce the following norms:
	\begin{equation*}\label{eq1.5}
		\|u\|_{*}=\sup _{x \in \Omega}\Bigl(\sigma(x) \sum_{j=0}^{\infty} \frac{\lambda^{\frac{N-2}{2}}}{\left(1+\lambda\left|y\right|+\lambda\left|z-\hat{z}_L^j\right|\right)^{\frac{N-2}{2}+\tau}}\Bigr)^{-1}|u(x)|
	\end{equation*}
	and
	\begin{equation*}\label{eq1.6}
		\|f\|_{**}=\sup _{x \in \Omega}\Bigl(\sigma(x) \sum_{j=0}^{\infty} \frac{\lambda^{\frac{N}{2}}}{\left|y\right|\left(1+\lambda\left|y\right|+\lambda\left|z-\hat{z}_L^j\right|\right)^{\frac{N}{2}+\tau}} \Bigr)^{-1}|f(x)|,
	\end{equation*}
	where $\hat{z}_L^{j}=\hat{z}-L\tilde{P}^{j}, \hat{z} \in B_1(0) \subset \mathbb{R}^{h}$, $\sigma(x)=\min \left\{1,\left(\frac{1+\lambda\left|y\right|+\lambda\left|z-\hat{z}\right|}{\lambda}\right)^{\tau}\right\}$, $\tau=\frac{N-2}{2}-\theta$ and $\theta>0$ is a fixed and sufficiently small constant. With this choice of $\theta$, it follows that $\tau>\bar{k}$. Here, we follow \cite{LWX} (see also \cite{Guo-Peng-Yan,Guo-Wu1,Guo-Wu2,Guo-Wu-Yan,Guo-Yan,WWZ}) to add an extra weight $\sigma(x)$ in $\|u\|_{*}$ in order to improve the estimates for the error term $\omega_{L}$ in the case $2 \leq \bar{k} \leq \tau$.
	
	Let
	$$
	\mathbf{X}=\left\{\phi \in L^{\infty}(\Omega): \phi \text { satisfies the boundary conditions in \eqref{eq1.1} and }\|\phi\|_{*}<+\infty\right\}
	$$	
	and
	$$
	\mathbf{Y}=\left\{f \in L^{\infty}(\Omega):\|f\|_{**}<+\infty\right\}.
	$$
	
	To carry out the reduction arguments in $\mathbf{X}$, we first study the invertibility of the linear operator
	\begin{equation}\label{eq1.7}
		\tilde{L} \phi:=\phi-\left(2^{\sharp}-1\right)(-\Delta)^{-1} \Bigl[\frac{M(x)\left(PW_{\hat{x},\lambda}\right)^{2^{\sharp}-2} \phi}{|y|}\Bigr], \ \phi \in \mathbf{X},
	\end{equation}
	where the linear operator $(-\Delta)^{-1}$ satisfies
	$$
	(-\Delta)^{-1} f(x)=\int_{\Omega} G(x, \xi) f(\xi) \mathrm{d}\xi
	$$	
	for any $f \in \mathbf{Y}$. To prove the invertibility of the above linear operator, we will use the Fredholm theory. This requires us to show that $(-\Delta)^{-1} \Bigl[\frac{M(x)\left(PW_{\hat{x},\lambda}\right)^{2^{\sharp}-2} \phi}{|y|}\Bigr]$ is compact in $\mathbf{X}$. This compactness allows us to apply the contraction mapping theorem, which in turn proves the existence of the correction term $\omega_{L}$.
	
	The paper is organized as follows. In Section \ref{sec2}, we discuss the invertibility of the linear operator defined in \eqref{eq1.7}. Section \ref{sec3} is devoted to the proof of Theorem 1.1. In the Appendix \ref{secA}, we provide some essential estimates for the approximate solutions $PW_{\hat{x},\lambda}$. We would like to point out that in order to find the correction term $\omega_{L}$, we have to estimate the approximate solutions $PW_{\hat{x},\lambda}$. We believe that these estimates are interesting on their own. Also, they can be helpful in other problems which are related to harmonic equations and periodic boundary conditions.

	\section{The invertibility of the linear operator} \label{sec2}
	In this section, we prove the invertibility of the linear operator defined in \eqref{eq1.7}. Recall that we aim to find a solution of the form $PW_{\hat{x}_{0,L}, \lambda_{L}}+\omega_{L}$ for \eqref{eq1.1}. A direct computation then shows that $\omega_{L}$ satisfies
	\begin{equation}\label{eq2.1}
		-\Delta \omega_{L}-\left(2^{\sharp}-1\right) \frac{M(x) (PW_{\hat{x},\lambda})^{2^{\sharp}-2}\omega_{L}}{|y|}=l_{L}+N\left(\omega_{L}\right),
	\end{equation}
	where
	\begin{equation*}\label{eq2.2}
		l_{L}=\frac{M(x) (PW_{\hat{x}_{0,L}, \lambda_{L}})^{2^{\sharp}-1}-W_{\hat{x}_{0,L}, \lambda_{L}}^{2^{\sharp}-1}}{|y|}
	\end{equation*}
	and
	\begin{equation*}\label{eq2.3}
		N\left(\omega_{L}\right)=\frac{M(x)\left[\left(PW_{\hat{x}_{0,L}, \lambda_{L}}+\omega_{L}\right)_{+}^{2^{\sharp}-1}-(PW_{\hat{x}_{0,L}, \lambda_{L}})^{2^{\sharp}-1}-\left(2^{\sharp}-1\right) (PW_{\hat{x}_{0,L}, \lambda_{L}})^{2^{\sharp}-2} \omega_{L}\right]}{|y|}.
	\end{equation*}

	Throughout this paper, we always assume that 
	\begin{equation} \label{aeq2.4}
		|\hat{x}| \leq \frac{1}{\lambda^{1+\theta}}, \ \frac{C_{0}}{2} \leq \lambda L^{-\frac{N-2}{\beta-N+2}} \leq 2 C_{0},
	\end{equation} 
	for some $C_{0}>0$, where $\theta>0$ is a sufficiently small constant. We also introduce the following notations
	\begin{equation} \label{eq2.4}
		\partial_{i}\left(P W_{\hat{x},\lambda}\right)=\begin{cases}
			\frac{\partial P W_{\hat{x},\lambda}}{\partial \hat{z}_{i}} & \text { if } i=1, \cdots, h, \\
			\frac{\partial P W_{\hat{x},\lambda}}{\partial \lambda} & \text { if } i=h+1,
		\end{cases}
	\end{equation} 
	\begin{equation}\label{eq2.5}
		\alpha(i)=\begin{cases}
			1 & \text { if } i=1, \cdots, h, \\
			-1 & \text { if } i=h+1
		\end{cases}
	\end{equation} 
	and
	\begin{equation}\label{eq2.6}
		\psi_{0}=\left.\frac{\partial W_{0, \lambda}}{\partial \lambda}\right|_{\lambda=1}, \ \psi_{i}=\frac{\partial W_{0,1}}{\partial \hat{z}_{i}}, \ i=1, \cdots, h.
	\end{equation}

	For any $f \in \mathbf{Y}$, we recall that 
	$$
	(-\Delta)^{-1} f:=\int_{\Omega} G(x, \xi) f(\xi) \mathrm{d}\xi.
	$$
	
	\begin{lemma}\label{lm2.1}
		$(-\Delta)^{-1}$ is a bounded linear operator from $\mathbf{Y}$ to $\mathbf{X}$.
	\end{lemma}
	\begin{proof}
		We extend the function $\sigma(\xi)$ periodically to $\mathbb{R}^{N}$. For any $f \in \mathbf{Y}$, then we have
		\begin{align*}
			|u(x)|=\left|(-\Delta)^{-1} f\right| & \leq\|f\|_{**} \int_{\Omega} G(x, \xi) \sigma(\xi) \sum_{j=0}^{\infty}\frac{\lambda^{\frac{N}{2}}}{\left|\xi^1\right|\left(1+\lambda\left|\xi^1\right|+\lambda\left|\xi^2-\hat{z}_L^j\right|\right)^{\frac{N}{2}+\tau}} \mathrm{d}\xi \\
			& =\|f\|_{**} \int_{\Omega} \sum_{j=0}^{\infty} \Gamma\left(x, \xi+P^{j}_L\right) \sigma(\xi) \sum_{j=0}^{\infty}\frac{\lambda^{\frac{N}{2}}}{\left|\xi^1\right|\left(1+\lambda\left|\xi^1\right|+\lambda\left|\xi^2-\hat{z}_L^j\right|\right)^{\frac{N}{2}+\tau}} \mathrm{d}\xi \\
			& =\|f\|_{**} \int_{\mathbb{R}^{N}} \Gamma(x, \xi) \sigma(\xi) \sum_{j=0}^{\infty}\frac{\lambda^{\frac{N}{2}}}{\left|\xi^1\right|\left(1+\lambda\left|\xi^1\right|+\lambda\left|\xi^2-\hat{z}_L^j\right|\right)^{\frac{N}{2}+\tau}} \mathrm{d}\xi \\
			& \leq C\|f\|_{**} \sigma(x) \sum_{j=0}^{\infty}\frac{\lambda^{\frac{N-2}{2}}}{\left(1+\lambda\left|y\right|+\lambda\left|z-\hat{z}_L^j\right|\right)^{\frac{N-2}{2}+\tau}},
		\end{align*}
		where we use Lemma B.2 of \cite{WWY} and $N \geq 5$ in the last inequality.
		This gives that $\|u\|_{*} \leq C\|f\|_{**}.$ So $(-\Delta)^{-1}$ is a bounded linear operator from $\mathbf{Y}$ to $\mathbf{X}$.
	\end{proof}
	
	We rewrite \eqref{eq2.1} as 
	\begin{equation}\label{eq2.7}
		\omega_{L}-\left(2^{\sharp}-1\right)(-\Delta)^{-1}\Bigl(\frac{M(x) (PW_{\hat{x},\lambda})^{2^{\sharp}-2}\omega_{L}}{|y|}\Bigr)=(-\Delta)^{-1} l_{L}+(-\Delta)^{-1} N\left(\omega_{L}\right).
	\end{equation}
	For $\hat{z} \in B_1(0) \subset \mathbb{R}^{h}$ and $\lambda>0$ sufficiently large, we denote
	\begin{equation*}\label{eq2.8}
		Z_{i}=\begin{cases}
			\frac{\partial W_{\hat{x},\lambda}}{\partial \hat{z}_{i}} & \text { if } i=1, \cdots, h, \\
			\frac{\partial W_{\hat{x},\lambda}}{\partial \lambda} & \text { if } i=h+1.
		\end{cases}
	\end{equation*}
	Noting that for any $f \in \mathbf{Y}$ and $u=(-\Delta)^{-1} f \in \mathbf{X}$, we have
	\begin{equation*}
		\begin{aligned}
			\int_{\Omega} f \partial_{i}\left(P W_{\hat{x},\lambda}\right) & =-\int_{\Omega}\Delta u \partial_{i}\left(P W_{\hat{x},\lambda}\right) =-\int_{\Omega}\Delta\left(\partial_{i}\left(P W_{\hat{x},\lambda}\right)\right) u 
			=\left(2^{\sharp}-1\right) \int_{\Omega} \frac{u W_{\hat{x},\lambda}^{2^{\sharp}-2} Z_{i}}{|\xi^1|},
		\end{aligned}
	\end{equation*}
	where $\partial_{i}\left(P W_{\hat{x},\lambda}\right)$ is defined in \eqref{eq2.4}.
	Let
	\begin{equation*}\label{eq2.9}
		\mathbf{E}:=\left\{\phi \in \mathbf{X} \,: \int_{\Omega} \frac{\phi W_{\hat{x},\lambda}^{2^{\sharp}-2} Z_{i}}{|\xi^1|}=0,\, i=1, \cdots, h+1\right\}
	\end{equation*} 
	and
	\begin{equation*}\label{eq2.10}
		\mathbf{F}:=\left\{f \in \mathbf{Y} \,: \int_{\Omega} f \partial_{i}\left(P W_{\hat{x},\lambda}\right)=0,\, i=1, \cdots, h+1\right\}.
	\end{equation*} 
	Then, we see that $f \in \mathbf{F}$ if and only if $u=(-\Delta)^{-1} f \in \mathbf{E}$.
	
	We aim to find a solution of the form $P W_{\hat{x}_{0,L}, \lambda_{L}}+\omega_{L}$ for \eqref{eq1.1}, where $\omega_{L} \in \mathbf{E}$ and $\left\|\omega_{L}\right\|_{*}$ is sufficiently small. To achieve this goal, we first prove that for fixed $(\hat{x},\lambda)$, there exists a smooth function 
	$\omega \in \mathbf{E}$, such that
	\begin{equation*}\label{eq2.11}
		-\Delta\left(P W_{\hat{x},\lambda}+\omega\right)-\frac{M(x)\left(P W_{\hat{x},\lambda}+\omega\right)^{2^{\sharp}-1}}{|y|}=\sum_{i=1}^{h+1} c_{i} \frac{W_{\hat{x},\lambda}^{2^{\sharp}-2} Z_{i}}{|y|}
	\end{equation*} 
	for some constants $c_{i}$. Then, we demonstrate the existence of $(\hat{x},\lambda)=\left(\hat{x}_{0,L}, \lambda_{L}\right)$ such that
	\begin{equation*}\label{eq2.12}
		\int_{\Omega}-\Delta\left(P W_{\hat{x},\lambda}+\omega\right) \partial_{i}\left(P W_{\hat{x},\lambda}\right)-\int_{\Omega} \frac{M(\xi)\left(P W_{\hat{x},\lambda}+\omega\right)^{2^{\sharp}-1}\partial_{i}\left(P W_{\hat{x},\lambda}\right)}{|\xi^1|}=0,
	\end{equation*} 
	for $i=1, \cdots, h+1$. With this $\left(\hat{x}_{0,L}, \lambda_{L}\right)$, we can prove that all $c_{i}$ must be zero.
	
	We first define the operator $\mathbf{P}$ as follows:
	\begin{equation*}\label{eq2.13}
		\mathbf{P} f=f+\sum_{i=1}^{h+1} c_{i} \frac{W_{\hat{x},\lambda}^{2^{\sharp}-2} Z_{i}}{|y|}, \ f \in \mathbf{Y},
	\end{equation*} 
	where $c_{i}$ is chosen such that $\mathbf{P} f \in \mathbf{F}$. Then it is straightforward to verify that $\mathbf{P}$ is a bounded operator, that is 
	$$
	\|\mathbf{P} f\|_{**} \leq C\|f\|_{**},
	$$ 
	for some constant $C>0$.
	
	In the following, for a fixed $(\hat{x},\lambda)$, instead of \eqref{eq2.7}, we consider the following problem
	\begin{equation*}\label{eq2.14}
		\omega-T \omega=(-\Delta)^{-1}\left(\mathbf{P} l\right)+(-\Delta)^{-1}(\mathbf{P} N(\omega)),
	\end{equation*} 
	where
	\begin{equation}\label{eq2.15}
		T\omega=\left(2^{\sharp}-1\right)(-\Delta)^{-1}\Bigl[\mathbf{P}\Bigl(\frac{M(x)\left(P W_{\hat{x},\lambda}\right)^{2^{\sharp}-2} \omega}{|y|}\Bigr)\Bigr].
	\end{equation} 
	
	We now prove that the linear operator $I-T$ is invertible. According to Fredholm theorem, it is sufficient to prove that $T$ is a bounded compact operator and $I-T$ is injective. We begin with the following result.
	\begin{proposition}\label{pro2.2}
		$I-T$ is a bijective bounded linear operator from $\mathbf{E}$ to itself.
	\end{proposition}
	
	We have
	\begin{equation*}\label{eq2.16}
		\mathbf{P}\Bigl(\frac{M(x)\left(P W_{\hat{x},\lambda}\right)^{2^{\sharp}-2} \phi}{|y|}\Bigr)=\frac{M(x)\left(P W_{\hat{x},\lambda}\right)^{2^{\sharp}-2} \phi}{|y|}+\sum_{i=1}^{h+1} c_{i} \frac{W_{\hat{x},\lambda}^{2^{\sharp}-2} Z_{i}}{|y|}, \ \phi \in \mathbf{E},
	\end{equation*} 
	where the constants $c_{i}$ are determined by
	\begin{equation}\label{eq2.17}
		\begin{aligned}
			\sum_{i=1}^{h+1} c_{i} \int_{\Omega} \frac{W_{\hat{x},\lambda}^{2^{\sharp}-2} Z_{i}\partial_{h}\left(P W_{\hat{x},\lambda}\right)}{|\xi^1|} =-\int_{\Omega} \frac{M(\xi)\left(P W_{\hat{x},\lambda}\right)^{2^{\sharp}-2} \phi\partial_{h}\left(P W_{\hat{x},\lambda}\right)}{|\xi^1|}, \ i=1, \cdots, h+1.
		\end{aligned}
	\end{equation}
	
	We divide the proof of Proposition \ref{pro2.2} into the following lemmas.
	\begin{lemma}\label{lm2.3}
		We have
		\begin{equation}\label{eq2.18}
			\Bigl|\int_{\Omega} \frac{M(\xi)\left(P W_{\hat{x},\lambda}\right)^{2^{\sharp}-2} \omega\partial_{i}\left(P W_{\hat{x},\lambda}\right)}{|\xi^1|}\Bigr| \leq \frac{C \lambda^{\alpha(i)}}{\lambda^{\beta}}\|\omega\|_{*}, \ \omega \in \mathbf{E}
		\end{equation} 	
		and
		\begin{equation}\label{eq2.19}
			|c_{i}| \leq \frac{C}{\lambda^{\alpha(i)+\beta}}\|\omega\|_{*},
		\end{equation} 	
		where $\alpha(i)$ is defined in \eqref{eq2.5}.
	\end{lemma}
	
	\begin{proof}
		Write
		\begin{equation*}\label{eq2.20}
			\begin{aligned}
				& \quad \int_{\Omega} \frac{M(\xi)\left(P W_{\hat{x},\lambda}\right)^{2^{\sharp}-2} \omega\partial_{i}\left(P W_{\hat{x},\lambda}\right)}{|\xi^1|} \\
				& =\int_{\Omega} \frac{M(\xi)\left[\left(P W_{\hat{x},\lambda}\right)^{2^{\sharp}-2}- W_{\hat{x},\lambda}^{2^{\sharp}-2}\right] \omega\partial_{i}\left(P W_{\hat{x},\lambda}\right)}{|\xi^1|} \\
				& \quad +\int_{\Omega} \frac{(M(\xi)-1)W_{\hat{x},\lambda}^{2^{\sharp}-2} \omega\partial_{i}\left(P W_{\hat{x},\lambda}\right)}{|\xi^1|} +\int_{\Omega} \frac{W_{\hat{x},\lambda}^{2^{\sharp}-2} \omega\partial_{i}\left(P W_{\hat{x},\lambda}\right)}{|\xi^1|} \\
				& :=J_1+J_2+J_3.
			\end{aligned}
		\end{equation*} 	
		By Lemma \ref{lmA.1}, it follows that
		\begin{align*}
			\left|J_1\right| & \leq C \lambda^{\alpha(i)}\int_{\Omega} \frac{\left|\left(P W_{\hat{x},\lambda}\right)^{2^{\sharp}-2}- W_{\hat{x},\lambda}^{2^{\sharp}-2}\right| |\omega| P W_{\hat{x},\lambda}}{|\xi^1|} \\
			& \leq C\|\omega\|_{*} \lambda^{\alpha(i)} \int_{\Omega} \frac{\left|\left(P W_{\hat{x},\lambda}\right)^{2^{\sharp}-2}- W_{\hat{x},\lambda}^{2^{\sharp}-2}\right| P W_{\hat{x},\lambda}}{|\xi^1|}\sum_{j=0}^{\infty} \frac{\lambda^{\frac{N-2}{2}}}{\left(1+\lambda\left|\xi^1\right|+\lambda\left|\xi^2-\hat{z}_L^j\right|\right)^{\frac{N-2}{2}+\tau}} \\
			& \leq C\|\omega\|_{*} \lambda^{\alpha(i)} \int_{\Omega} \frac{\left(P W_{\hat{x},\lambda}\right)^{2^{\sharp}-2}|\varphi_{\hat{x},\lambda}|}{|\xi^1|}\sum_{j=0}^{\infty} \frac{\lambda^{\frac{N-2}{2}}}{\left(1+\lambda\left|\xi^1\right|+\lambda\left|\xi^2-\hat{z}_L^j\right|\right)^{\frac{N-2}{2}+\tau}} \\
			& \leq C\|\omega\|_{*} \lambda^{\alpha(i)} \int_{\Omega} \frac{\left(\sum_{j=0}^{\infty} W_{\hat{x}_{L}^{j}, \lambda}\right)^{2^{\sharp}-2}\left|\varphi_{\hat{x},\lambda}\right|}{|\xi^1|} \sum_{j=0}^{\infty}\frac{\lambda^{\frac{N-2}{2}}}{\left(1+\lambda\left|\xi^1\right|+\lambda\left|\xi^2-\hat{z}_L^j\right|\right)^{\frac{N-2}{2}+\tau}},
		\end{align*} 	
		where
		\begin{equation*}\label{eq2.22}
			\varphi_{\hat{x},\lambda}=W_{\hat{x},\lambda}-P W_{\hat{x},\lambda}.
		\end{equation*}
		For $\xi \in B_1(\hat{x})$, we have
		\begin{equation}\label{eq2.23}
			\begin{aligned}
				\sum_{j=0}^{\infty} W_{\hat{x}_{L}^{j}, \lambda} & \leq \frac{C \lambda^{\frac{N-2}{2}}}{\left(1+\lambda\left|\xi^1\right|+\lambda\left|\xi^2-\hat{z}\right|\right)^{N-2}}+C \sum_{j=1}^{\infty} \frac{1}{\lambda^{\frac{N-2}{2}}\left|\xi^2-\hat{z}^{j}_L\right|^{N-2 }} \\
				& \leq \frac{C \lambda^{\frac{N-2}{2}}}{\left(1+\lambda\left|\xi^1\right|+\lambda\left|\xi^2-\hat{z}\right|\right)^{N-2}}+\frac{C}{\lambda^{\frac{N-2}{2}} L^{N-2}} \\
				& \leq \frac{C \lambda^{\frac{N-2}{2}}}{\left(1+\lambda\left|\xi^1\right|+\lambda\left|\xi^2-\hat{z}\right|\right)^{N-2}}
			\end{aligned}
		\end{equation}
		and
		\begin{equation}\label{eq2.24}
			\begin{aligned}
				\sum_{j=0}^{\infty}\frac{\lambda^{\frac{N-2}{2}}}{\left(1+\lambda\left|\xi^1\right|+\lambda\left|\xi^2-\hat{z}_L^j\right|\right)^{\frac{N-2}{2}+\tau}} & \leq C\lambda^{\frac{N-2}{2}}\Bigl(\frac{1}{\left(1+\lambda\left|\xi^1\right|+\lambda\left|\xi^2-\hat{z}\right|\right)^{\frac{N-2}{2}+\tau}}+\frac{1}{(\lambda L)^{\frac{N-2}{2}+\tau}}\Bigr) \\
				& \leq \frac{C \lambda^{\frac{N-2}{2}}}{(1+\lambda\left|\xi^1\right|+\lambda\left|\xi^2-\hat{z}\right|)^{\frac{N-2}{2}+\tau}}.
			\end{aligned}
		\end{equation} 	
		Combining \eqref{eq2.23}, \eqref{eq2.24} with Lemma \ref{lmA.2}, we obtain
		\begin{equation}\label{eq2.25}
			\begin{aligned}
				& \quad \int_{B_1(\hat{x})} \frac{\Bigl(\sum_{j=0}^{\infty} W_{\hat{x}_{L}^{j}, \lambda}\Bigr)^{2^{\sharp}-2}\left|\varphi_{\hat{x},\lambda}\right|}{|\xi^1|} \sum_{j=0}^{\infty}\frac{\lambda^{\frac{N-2}{2}}}{\left(1+\lambda\left|\xi^1\right|+\lambda\left|\xi^2-\hat{z}_L^j\right|\right)^{\frac{N-2}{2}+\tau}} \\
				& \leq \frac{C}{\lambda^{\frac{N-2}{2}} L^{N-2}} \int_{B_1(\hat{x})}\frac{\Bigl(\frac{\lambda^{\frac{N-2}{2}}}{\left(1+\lambda\left|\xi^1\right|+\lambda\left|\xi^2-\hat{z}\right|\right)^{N-2}}\Bigr)^{2^{\sharp}-2}}{\left|\xi^1\right|}\frac{ \lambda^{\frac{N-2}{2}}}{(1+\lambda\left|\xi^1\right|+\lambda\left|\xi^2-\hat{z}\right|)^{\frac{N-2}{2}+\tau}} \\
				& \leq \frac{C}{\lambda^{N-2} L^{N-2}}.
			\end{aligned}
		\end{equation} 	
		On the other hand, for $\xi \in \Omega \backslash B_1(\hat{x})$, it holds that
		\begin{align}\label{eqs2.26} 
			& \quad \int_{\Omega \backslash B_1(\hat{x})} \frac{\left(\sum_{j=0}^{\infty} W_{\hat{x}_{L}^{j}, \lambda}\right)^{2^{\sharp}-2}\left|\varphi_{\hat{x},\lambda}\right|}{|\xi^1|} \sum_{j=0}^{\infty}\frac{\lambda^{\frac{N-2}{2}}}{\left(1+\lambda\left|\xi^1\right|+\lambda\left|\xi^2-\hat{z}_L^j\right|\right)^{\frac{N-2}{2}+\tau}} \nonumber\\
			& \leq \int_{\Omega \backslash B_1(\hat{x})} \frac{\left(\sum_{j=0}^{\infty} W_{\hat{x}_{L}^{j}, \lambda}\right)^{2^{\sharp}-1}}{|\xi^1|} \sum_{j=0}^{\infty}\frac{\lambda^{\frac{N-2}{2}}}{\left(1+\lambda\left|\xi^1\right|+\lambda\left|\xi^2-\hat{z}_L^j\right|\right)^{\frac{N-2}{2}+\tau}}. 
		\end{align} 
		We have
		\begin{equation*}\label{eqs2.27} 
			\begin{aligned}
				& \quad \int_{\Omega \backslash B_1(\hat{x})} \frac{\lambda^{\frac{N-2}{2}}}{\left(1+\lambda |\xi^1|+\lambda|\xi^2-\hat{z}|\right)^{\frac{N-2}{2}+\tau}} \frac{W_{\hat{x},\lambda}^{2^{\sharp}-1}}{|\xi^1|} \\
				& \leq \int_{\mathbb{R}^{N} \backslash B_{\lambda}(0)} \frac{1}{|\xi^1|(1+|\xi^1|+|\xi^2|)^{N+\frac{N-2}{2}+\tau}} \leq \frac{C}{\lambda^{N-1-\theta}}
			\end{aligned}
		\end{equation*}
		and
		\begin{equation*}\label{eqs2.28}
			\begin{aligned}
				& \quad \int_{\Omega \backslash B_1(\hat{x})} \sum_{j=1}^{\infty}\frac{\lambda^{\frac{N-2}{2}}}{\left(1+\lambda\left|\xi^1\right|+\lambda\left|\xi^2-\hat{z}_L^j\right|\right)^{\frac{N-2}{2}+\tau}} \frac{W_{\hat{x},\lambda}^{2^{\sharp}-1}}{|\xi^1|} \\
				& \leq \frac{C \lambda^{N-1}}{(\lambda L)^{\frac{N-2}{2}+\tau}} \int_{\Omega \backslash B_1(\hat{x})} \frac{1}{|\xi^1|(1+\lambda |\xi^1|+\lambda|\xi^2-\hat{z}|)^{N}} \\
				& \leq \frac{C}{(\lambda L)^{\frac{N-2}{2}+\tau}} \int_{\mathbb{R}^{N} \backslash B_{\lambda}(0)} \frac{1}{|\xi^1|(1+|\xi^1|+|\xi^2-\hat{z}|)^{N}} \leq \frac{C}{\lambda (\lambda L)^{N-2-\theta}},
			\end{aligned}
		\end{equation*}
		since $\tau=\frac{N-2}{2}-\theta$.
		
		From 
		\begin{equation*}
			\begin{aligned}
				\sum_{j=1}^{\infty} W_{\hat{x}_{L}^{j}, \lambda} & \leq \frac{C \lambda^{\frac{N-2}{2}}}{(1+\lambda \left|\xi^1\right|+\lambda \left|\xi^2-\hat{z}\right|)^{(2+2 \theta) /\left(2^{\sharp}-1\right)}} \sum_{j=1}^{\infty} \frac{1}{\left(1+\lambda\left|\xi^1\right|+\lambda\left|\xi^2-\hat{z}_L^j\right|\right)^{N-2-(2+2 \theta) /\left(2^{\sharp}-1\right)}} \\
				& \leq \frac{C \lambda^{\frac{N-2}{2}}}{(1+\lambda \left|\xi^1\right|+\lambda \left|\xi^2-\hat{z}\right|)^{(2+2 \theta) /\left(2^{\sharp}-1\right)}} \frac{1}{(\lambda L)^{N-2-(2+2 \theta) /\left(2^{\sharp}-1\right)}},
			\end{aligned}
		\end{equation*}
		we obtain
		\begin{equation}\label{eqs2.29}
			\begin{aligned}
				& \quad \int_{\Omega \backslash B_1(\hat{x})} \frac{\lambda^{\frac{N-2}{2}}}{\left(1+\lambda |\xi^1|+\lambda|\xi^2-\hat{z}|\right)^{\frac{N-2}{2}+\tau}} \frac{\left(\sum_{j=1}^{\infty} W_{\hat{x}_{L}^{j}, \lambda}\right)^{2^{\sharp}-1}}{|\xi^1|} \\
				& \leq \int_{\mathbb{R}^{N} \backslash B_{\lambda}(0)} \frac{1}{|\xi^1|(1+|\xi^1|+|\xi^2|)^{2+2 \theta+\frac{N-2}{2}+\tau}} \frac{1}{(\lambda L)^{\left(2^{\sharp}-1\right)(N-2)-(2+2 \theta)}} \leq \frac{C}{\lambda^{1+\theta} (\lambda L)^{N-2-2 \theta}}.
			\end{aligned}
		\end{equation}
		Similarly, we can prove
		\begin{align}\label{eqs2.30} 
			& \quad \int_{\Omega \backslash B_1(\hat{x})} \sum_{j=1}^{\infty}\frac{\lambda^{\frac{N-2}{2}}}{\left(1+\lambda\left|\xi^1\right|+\lambda\left|\xi^2-\hat{z}_L^j\right|\right)^{\frac{N-2}{2}+\tau}} \frac{\left(\sum_{j=1}^{\infty} W_{\hat{x}_{L}^{j}, \lambda}\right)^{2^{\sharp}-1}}{|\xi^1|} \nonumber\\
			& \leq C \lambda^{N-1} \int_{\Omega \backslash B_1(\hat{x})} \Bigl[\frac{1}{|\xi^1| (1+\lambda |\xi^1|+\lambda|\xi^2-\hat{z}|)^{\frac{N-2}{2}}} \frac{1}{(\lambda L)^{\tau}} \\
			& \quad \quad \quad \quad \quad \quad \quad \quad \Bigl(\frac{1}{(1+\lambda |\xi^1|+\lambda|\xi^2-\hat{z}|)^{\frac{N-2}{2}+\frac{\theta}{2^{\sharp}-1}}} \frac{1}{(\lambda L)^{\frac{N-2}{2}-\frac{\theta}{2^{\sharp}-1}}}\Bigr)^{2^{\sharp}-1}\Bigr] \nonumber\\
			& \leq C \int_{\mathbb{R}^{N} \backslash B_{\lambda}(0)} \frac{1}{|\xi^1| (1+|\xi^1|+|\xi^2|)^{N-1+\theta}} \frac{1}{(\lambda L)^{\tau+\frac{N}{2}-\theta}} \leq \frac{C}{\lambda^{\theta} (\lambda L)^{N-1-2 \theta}}. \nonumber
		\end{align}
		Combining \eqref{eqs2.26}-\eqref{eqs2.30}, we find that
		\begin{equation*}\label{eqs2.31}
			\int_{\Omega \backslash B_1(\hat{x})} \frac{\left(\sum_{j=0}^{\infty} W_{\hat{x}_{L}^{j}, \lambda}\right)^{2^{\sharp}-2}\left|\varphi_{\hat{x},\lambda}\right|}{|\xi^1|} \sum_{j=0}^{\infty}\frac{\lambda^{\frac{N-2}{2}}}{\left(1+\lambda\left|\xi^1\right|+\lambda\left|\xi^2-\hat{z}_L^j\right|\right)^{\frac{N-2}{2}+\tau}} \leq \frac{C}{\lambda^{N-1-\theta}},
		\end{equation*}
		which together with \eqref{eq2.25} gives
		\begin{equation}\label{eqs2.32}
			|J_1|\leq \frac{C\lambda^{\alpha(i)}\|\omega\|_{*}}{\lambda^{\beta}}.
		\end{equation}
		To estimate $J_2$, using \eqref{eq2.23} and \eqref{eq2.24}, similar to the proof of \eqref{eqs2.26}, we find that
		\begin{equation}\label{eqs2.33}
			\begin{aligned}
				|J_2|&=\Bigl|\int_{\Omega} \frac{(M(\xi)-1)W_{\hat{x},\lambda}^{2^{\sharp}-2} \omega\partial_{i}\left(P W_{\hat{x},\lambda}\right)}{|\xi^1|}\Bigr|\\
				& \leq C\lambda^{\alpha(i)}\|\omega\|_{*} \int_{\Omega} \frac{|M(\xi)-1| W_{\hat{x},\lambda}^{2^{\sharp}-2} }{|\xi^1|} \sum_{j=0}^{\infty} \frac{\lambda^{\frac{N-2}{2}}}{\left(1+\lambda|\xi^1|+\lambda\left|\xi^2-\hat{z}_L^{j}\right|\right)^{\frac{N-2}{2}+\tau}} \sum_{j=0}^{\infty} W_{\hat{x}_L^{j}, \lambda} \\
				& \leq C\lambda^{\alpha(i)}\|\omega\|_{*} \Bigl(\lambda^{N-1}\int_{B_1(\hat{x})}|\xi|^{\beta} \frac{1}{|\xi^1| (1+\lambda|\xi^1|+\lambda|\xi^2-\hat{z}|)^{2+\frac{N-2}{2}+\tau+N-2}} \\
				&\quad + \int_{\Omega \backslash B_1(\hat{x})} \frac{ W_{\hat{x},\lambda}^{2^{\sharp}-2} }{|\xi^1|} \sum_{j=0}^{\infty} \frac{\lambda^{\frac{N-2}{2}}}{\left(1+\lambda|\xi^1|+\lambda\left|\xi^2-\hat{z}_L^{j}\right|\right)^{\frac{N-2}{2}+\tau}} \sum_{j=0}^{\infty} W_{\hat{x}_L^{j}, \lambda}\Bigr) \\
				& \leq C\lambda^{\alpha(i)}\|\omega\|_{*} \Bigl(\frac{\left(1+|\lambda \hat{x}|^{\beta}\right)}{\lambda^{\beta}}+\frac{1}{\lambda^{N-1-\theta}}\Bigr) \leq \frac{C\lambda^{\alpha(i)}\|\omega\|_{*}}{\lambda^{\beta}}, 
			\end{aligned}
		\end{equation} 
		where in the last inequality, we use $|\hat{x}|=o\left(\frac{1}{\lambda}\right)$ in the assumption \eqref{aeq2.4}.
		
		To estimate $J_3$, noting that $\omega \in \mathbf{E}$, it follows by Lemma \ref{lmA.1} that
		\begin{equation}\label{eqs2.34}
			\begin{aligned}
				\left|J_3\right| & =\Bigl|\int_{\Omega} \frac{W_{\hat{x},\lambda}^{2^{\sharp}-2} \omega (\partial_{i}\left(P W_{\hat{x},\lambda}\right)-\partial_{i} W_{\hat{x},\lambda})}{|\xi^1|}\Bigr| \leq C \lambda^{\alpha(i)} \int_{\Omega} \frac{W_{\hat{x},\lambda}^{2^{\sharp}-2}|\omega| \sum_{j=1}^{\infty} W_{\hat{x}_L^{j}, \lambda}}{|\xi^1|} \\
				& \leq C \lambda^{\alpha(i)}\|\omega\|_{*} \int_{\Omega} \frac{W_{\hat{x},\lambda}^{2^{\sharp}-2}}{|\xi^1|} \sum_{j=0}^{\infty} \frac{\lambda^{\frac{N-2}{2}}}{\left(1+\lambda|\xi^1|+\lambda\left|\xi^2-\hat{z}_L^{j}\right|\right)^{\frac{N-2}{2}+\tau}} \sum_{j=1}^{\infty} W_{\hat{x}_L^{j}, \lambda}.
			\end{aligned}
		\end{equation}
		Using \eqref{eq2.23} and \eqref{eq2.24}, we obtain
		\begin{equation}\label{eqs2.35}
			\begin{aligned}
				& \quad \int_{B_1(\hat{x})} \frac{W_{\hat{x},\lambda}^{2^{\sharp}-2}}{|\xi^1|} \sum_{j=0}^{\infty} \frac{\lambda^{\frac{N-2}{2}}}{\left(1+\lambda|\xi^1|+\lambda\left|\xi^2-\hat{z}_L^{j}\right|\right)^{\frac{N-2}{2}+\tau}} \sum_{j=1}^{\infty} W_{\hat{x}_L^{j}, \lambda} \\
				& \leq C \lambda^{N-1} \int_{B_1(\hat{x})} \frac{1}{|\xi^1|(1+\lambda|\xi^1|+\lambda\left|\xi^2-\hat{z}_L^{j}\right|)^{\frac{N-2}{2}+\tau+2}} \frac{1}{(\lambda L)^{N-2}} \leq \frac{C}{(\lambda L)^{N-2}}.
			\end{aligned}
		\end{equation}
		Similar to the proof of \eqref{eqs2.26}, we can prove
		\begin{equation}\label{eqs2.36}
			\int_{\Omega \backslash B_1(\hat{x})} \frac{W_{\hat{x},\lambda}^{2^{\sharp}-2}}{|\xi^1|} \sum_{j=0}^{\infty} \frac{\lambda^{\frac{N-2}{2}}}{\left(1+\lambda|\xi^1|+\lambda\left|\xi^2-\hat{z}_L^{j}\right|\right)^{\frac{N-2}{2}+\tau}} \sum_{j=1}^{\infty} W_{\hat{x}_L^{j}, \lambda} \leq \frac{C}{\lambda^{N-1-\theta}}. 
		\end{equation}
		Thus by \eqref{eqs2.34}-\eqref{eqs2.36}, we have proved that
		\begin{equation}\label{eqs2.37}
			\left|J_3\right| \leq \frac{C \lambda^{\alpha(i)}\|\omega\|_{*}}{\lambda ^{\beta}}. 
		\end{equation}
		Therefore, \eqref{eq2.18} is derived from equations \eqref{eqs2.32}, \eqref{eqs2.33} and \eqref{eqs2.37}.
		
		Now we prove \eqref{eq2.19}. In fact, using Lemma \ref{lmA.2}, we have
		\begin{equation}\label{eqs2.38}
			\begin{aligned}
				\int_{\Omega} \frac{W_{\hat{x},\lambda}^{2^{\sharp}-2} Z_{i}\partial_{h}\left(P W_{\hat{x},\lambda}\right)}{|\xi^1|} & =\int_{\Omega} \frac{W_{\hat{x},\lambda}^{2^{\sharp}-2} Z_{i} Z_{h}}{|\xi^1|} +O\Bigl(\int_{\Omega} \frac{W_{\hat{x},\lambda}^{2^{\sharp}-2} |Z_{i}| |\partial_{h} \varphi_{\hat{x}, \lambda}|}{|\xi^1|}\Bigr) \\
				& =\int_{\mathbb{R}^{N}} \frac{W_{\hat{x},\lambda}^{2^{\sharp}-2} Z_{i}^2\left(\delta_{ih}+o(1)\right)}{|\xi^1|} \\ 
				& =\int_{\mathbb{R}^{N}} \frac{W_{\hat{x},\lambda}^{2^{\sharp}} \left(\delta_{ih}+o(1)\right)}{|\xi^1|}+O\Bigl(\int_{\mathbb{R}^{N}} \frac{W_{\hat{x},\lambda}^{2^{\sharp-2}}\left|Z_i^2-W_{\hat{x},\lambda}^2\right|}{|\xi^1|}\Bigr) \\
				& =\lambda^{2 \alpha(i)} a_{i}\left(\delta_{ih}+o(1)\right),
			\end{aligned}
		\end{equation}
		where $a_{i}>0$ is a constant, $\delta_{ih}=1$ if $i=h$ and $\delta_{ih}=0$ if $i \not = h$. So, we can solve \eqref{eq2.17} and use \eqref{eq2.18} to obtain \eqref{eq2.19}.
	\end{proof}
	
	To estimate $(-\Delta)^{-1}\Bigl[\frac{M(x) (P W_{\hat{x},\lambda})^{2^{\sharp}-2} \omega}{|y|}\Bigr]$, we need the following lemma. 
	\begin{lemma}\label{lm2.4}
		If $N \geq 5$, $1 \leq \bar{k}<\frac{N-2}{2}$, then for any $\phi \in \mathbf{X}$, there exists a constant $C>0$, such that for any $x \in \Omega$, it holds that
		\begin{equation}\label{eqs2.39}
			\begin{aligned}
				&\quad \Bigl|(-\Delta)^{-1}\Bigl[\frac{M(x) (P W_{\hat{x}, \lambda})^{2^{\sharp}-2} \phi}{|y|}\Bigr]\Bigr|\\
				& \leq C\|\phi\|_{*} \Bigl[\sigma(x) \sum_{j=0}^{\infty} \frac{\lambda^{\frac{N-2}{2}}}{\left(1+\lambda|y|+\lambda\left|z-\hat{z}_L^{j}\right|\right)^{\frac{N-2}{2}+\tau+\theta}}\\
				& \quad +\lambda^{-\frac{2 \theta}{N-2}} \sigma(x) \sum_{j=0}^{\infty} \frac{\lambda^{\frac{N-2}{2}}}{\left(1+\lambda|y|+\lambda\left|z-\hat{z}_L^{j}\right|\right)^{\frac{N-2}{2}+\tau}}\Bigr].
			\end{aligned}
		\end{equation}
	\end{lemma}
	\begin{proof}
		For any $\phi \in \mathbf{X}$, it holds that
		\begin{equation}\label{eqs2.40}
			\begin{aligned}
				&\quad \Bigl|(-\Delta)^{-1}\Bigl[\frac{M(x)\left(PW_{\hat{x},\lambda}\right)^{2^{\sharp}-2}\phi}{|y|}\Bigr]\Bigr|
				=\Bigl|\int_{\Omega} G(x,\xi) \frac{M(\xi)\left(PW_{\hat{x},\lambda}\right)^{2^{\sharp}-2}\phi}{|\xi^1|}\mathrm{d}\xi\Bigr|\\
				&\leq C\|\phi\|_* \int_{\Omega}\frac{G(x,\xi)}{|\xi^1|}\Bigl(\sum_{j=0}^{\infty}W_{\hat{x}_L^j,\lambda}\Bigr)^{2^\sharp-2} \sigma(\xi) \sum_{j=0}^{\infty} \frac{\lambda^{\frac{N-2}{2}}}{\left(1+\lambda\left|\xi^1\right|+\lambda\left|\xi^2-\hat{z}_L^j\right|\right)^{\frac{N-2}{2}+\tau}}\mathrm{d}\xi.
			\end{aligned}
		\end{equation}
		Suppose that $\xi \in B_{1}(\hat{x}) \cap \Omega$. Then recalling \eqref{eq2.23} and \eqref{eq2.24}, we have
		$$\sum_{j=0}^{\infty} W_{\hat{x}_{L}^{j}, \lambda} \leq \frac{C \lambda^{\frac{N-2}{2}}}{\left(1+\lambda\left|\xi^1\right|+\lambda\left|\xi^2-\hat{z}\right|\right)^{N-2}}$$
		and 
		$$\sum_{j=0}^{\infty}\frac{\lambda^{\frac{N-2}{2}}}{\left(1+\lambda\left|\xi^1\right|+\lambda\left|\xi^2-\hat{z}_L^j\right|\right)^{\frac{N-2}{2}+\tau}} \leq \frac{C \lambda^{\frac{N-2}{2}}}{(1+\lambda\left|\xi^1\right|+\lambda\left|\xi^2-\hat{z}\right|)^{\frac{N-2}{2}+\tau}}.$$
		Thus, we extend the function $\sigma(\xi)$ periodically to $\mathbb{R}^N$. It holds that
		\begin{equation}\label{eqs2.41}
			\begin{aligned}
				&\quad \int_{B_{1}(\hat{x}) \cap \Omega}\frac{G(x,\xi)}{|\xi^1|}\Bigl(\sum_{j=0}^{\infty}W_{\hat{x}_L^j,\lambda}\Bigr)^{2^\sharp-2} \sigma(\xi) \sum_{j=0}^{\infty} \frac{\lambda^{\frac{N-2}{2}}}{\left(1+\lambda\left|\xi^1\right|+\lambda\left|\xi^2-\hat{z}_L^j\right|\right)^{\frac{N-2}{2}+\tau}}\mathrm{d}\xi\\
				&\leq C \int_{B_{1}(\hat{x}) \cap \Omega} \frac{G(x,\xi)}{|\xi^1|} \frac{\lambda^{\frac{N}{2}}}{(1+\lambda\left|\xi^1\right|+\lambda\left|\xi^2-\hat{z}\right|)^{\frac{N-2}{2}+2+\tau}} \Bigl(\frac{1+\lambda\left|\xi^1\right|+\lambda\left|\xi^2-\hat{z}\right|}{\lambda}\Bigr)^{\tau}\\
				&\leq C \int_{\Omega} \frac{G(x,\xi)}{|\xi^1|} \frac{\lambda^{\frac{N}{2}-{\tau}}}{(1+\lambda\left|\xi^1\right|+\lambda\left|\xi^2-\hat{z}\right|)^{\frac{N-2}{2}+2}}\\
				&= C \int_{\mathbb{R}^{N}} \Gamma(x,\xi) \sum_{j=0}^{\infty} \frac{\lambda^{\frac{N}{2}-{\tau}}}{|\xi^1| (1+\lambda\left|\xi^1\right|+\lambda\left|\xi^2-\hat{z}_L^j\right|)^{\frac{N-2}{2}+2}} \\
				&\leq C \sum_{j=0}^{\infty} \frac{\lambda^{\frac{N-2}{2}-\tau}}{(1+\lambda\left|y\right|+\lambda\left|z-\hat{z}_L^j\right|)^{\frac{N}{2}}}.
			\end{aligned}
		\end{equation}
		Notice that we have
		$$
		\sigma(x) \geq C \Bigl(\frac{1+\lambda\left|y\right|+\lambda\left|z-\hat{z}\right|}{\lambda}\Bigr)^{\tau}, \quad x \in B_{1}(\hat{x}) \cap \Omega.
		$$
		Thus, for $x \in B_{1}(\hat{x}) \cap \Omega$, on the one hand, when $j=0$, there holds
		\begin{equation}\label{eqs2.42}
			\begin{aligned}
				&\quad \frac{\lambda^{\frac{N-2}{2}-\tau}}{(1+\lambda\left|y\right|+\lambda\left|z-\hat{z}\right|)^{\frac{N}{2}}} \\
				& =\Bigl(\frac{1+\lambda\left|y\right|+\lambda\left|z-\hat{z}\right|}{\lambda}\Bigr)^{\tau} \frac{\lambda^{\frac{N-2}{2}}}{(1+\lambda\left|y\right|+\lambda\left|z-\hat{z}\right|)^{{\frac{N}{2}}+\tau}} \\
				&\leq \frac{C \sigma(x) \lambda^{\frac{N-2}{2}}}{(1+\lambda\left|y\right|+\lambda\left|z-\hat{z}\right|) ^{\frac{N-2}{2}+\tau+\theta}}.
			\end{aligned}
		\end{equation}
		On the other hand, for $j>0$, we have
		\begin{equation}\label{eqs2.43}
			\sum_{j=1}^{\infty} \frac{\lambda^{\frac{N-2}{2}-\tau}}{(1+\lambda\left|y\right|+\lambda\left|z-\hat{z}_L^j\right|)^{\frac{N}{2}}} \leq \frac{C \lambda^{\frac{N-2}{2}-\tau}}{(\lambda L)^{\frac{N}{2}}}
		\end{equation}
		and
		\begin{align}\label{eqs2.44} 
			&\quad \frac{\sigma(x) \lambda^{\frac{N-2}{2}}}{(1+\lambda\left|y\right|+\lambda\left|z-\hat{z}\right|)^{\frac{N-2}{2}+\tau+\theta}} \nonumber\\
			&\geq C \Bigl(\frac{1+\lambda\left|y\right|+\lambda\left|z-\hat{z}\right|}{\lambda}\Bigr)^{\tau} \frac{\lambda^{\frac{N-2}{2}}}{(1+\lambda\left|y\right|+\lambda\left|z-\hat{z}\right|)^{\frac{N-2}{2}+\tau+\theta}} \\
			&=C \frac{\lambda^{\frac{N-2}{2}-\tau}}{(1+\lambda\left|y\right|+\lambda\left|z-\hat{z}\right|)^{\frac{N-2}{2}+\theta}} \nonumber\\
			&\geq C \frac{\lambda^{\frac{N-2}{2}-\tau}}{\left(1+2\lambda\right)^{\frac{N-2}{2}+\theta}} \geq C \frac{\lambda^{\frac{N-2}{2}-{\tau}}}{(\lambda L) ^{\frac{N}{2}-1+\theta}}. \nonumber 
		\end{align}
		Combining \eqref{eqs2.43} and \eqref{eqs2.44}, we see that for $x \in B_{1}(\hat{x}) \cap \Omega$,
		\begin{equation}\label{eqs2.45}
			\sum_{j=1}^{\infty} \frac{\lambda^{\frac{N-2}{2}-\tau}}{(1+\lambda\left|y\right|+\lambda\left|z-\hat{z}_L^j\right|)^{\frac{N}{2}}} \leq \frac{C \sigma(x) \lambda^{\frac{N-2}{2}}}{(1+\lambda\left|y\right|+\lambda\left|z-\hat{z}\right|)^{\frac{N-2}{2}+\tau+\theta}}.
		\end{equation}
		Combining \eqref{eqs2.41}, \eqref{eqs2.42} and \eqref{eqs2.45}, we obtain
		\begin{equation}\label{eqs2.46}
			\begin{aligned}
				&\quad \int_{B_{1}(\hat{x}) \cap \Omega}\frac{G(x,\xi)}{|\xi^1|}\Bigl(\sum_{j=0}^{\infty}W_{\hat{x}_L^j,\lambda}\Bigr)^{2^\sharp-2} \sigma(\xi) \sum_{j=0}^{\infty} \frac{\lambda^{\frac{N-2}{2}}}{\left(1+\lambda\left|\xi^1\right|+\lambda\left|\xi^2-\hat{z}_L^j\right|\right)^{\frac{N-2}{2}+\tau}}\mathrm{d}\xi\\
				&\leq C \sigma(x) \sum_{j=0}^{\infty} \frac{\lambda^{\frac{N-2}{2}}}{\left(1+\lambda|y|+\lambda\left|z-\hat{z}_L^{j}\right|\right)^{\frac{N-2}{2}+\tau+\theta}}.
			\end{aligned}
		\end{equation}
		If $\xi \in \Omega \backslash B_{1}(\hat{x})$, then
		$$\sum_{j=0}^{\infty} W_{\hat{x}_L^{j}, \lambda}(\xi) \leq \frac{C }{\lambda^{\frac{N-2}{2}-\tau}} \sum_{j=0}^{\infty} \frac{\lambda^{\frac{N-2}{2}}}{\left(1+\lambda|\xi^1|+\lambda\left|\xi^2-\hat{z}_L^{j}\right|\right)^{\frac{N-2}{2}+\tau}}.$$
		Therefore, applying H\"older inequality, we find that
		\begin{equation}\label{eqs2.47}
			\begin{aligned}
				&\quad \Bigl(\sum_{j=0}^{\infty}W_{\hat{x}_L^j,\lambda}\Bigr)^{2^\sharp-2} \sum_{j=0}^{\infty} \frac{\lambda^{\frac{N-2}{2}}}{\left(1+\lambda\left|\xi^1\right|+\lambda\left|\xi^2-\hat{z}_L^j\right|\right)^{\frac{N-2}{2}+\tau}}\\
				&\leq \frac{C}{\lambda^{1-\frac{2\tau}{N-2}}} \Bigl(\sum_{j=0}^{\infty} \frac{\lambda^{\frac{N-2}{2}}}{\left(1+\lambda|\xi^1|+\lambda\left|\xi^2-\hat{z}_L^{j}\right|\right)^{\frac{N-2}{2}+\tau}}\Bigr)^{2^\sharp-1}\\
				& = C\lambda^{-\frac{2\theta}{N-2}} \Bigl(\sum_{j=0}^{\infty} \frac{\lambda^{\frac{N-2}{2}}}{\left(1+\lambda|\xi^1|+\lambda\left|\xi^2-\hat{z}_L^{j}\right|\right)^{\frac{N-2}{2}+\frac{N-2}{N}\tau+\frac{2}{N}\tau}}\Bigr)^{\frac{N}{N-2}}\\
				& \leq C\lambda^{-\frac{2\theta}{N-2}} \sum_{j=0}^{\infty} \frac{\lambda^{\frac{N}{2}}}{\left(1+\lambda|\xi^1|+\lambda\left|\xi^2-\hat{z}_L^{j}\right|\right)^{\frac{N}{2}+\tau}} \Bigl(\sum_{j=0}^{\infty} \frac{1}{\left(1+\lambda|\xi^1|+\lambda\left|\xi^2-\hat{z}_L^{j}\right|\right)^{\tau}}\Bigr)^{\frac{2}{N-2}}\\
				& \leq C\lambda^{-\frac{2\theta}{N-2}} \sum_{j=0}^{\infty} \frac{\lambda^{\frac{N}{2}}}{\left(1+\lambda|\xi^1|+\lambda\left|\xi^2-\hat{z}_L^{j}\right|\right)^{\frac{N}{2}+\tau}}, \quad \forall \,\, \xi \in \Omega \backslash B_{1}(\hat{x}),
			\end{aligned}
		\end{equation}
		where we use 
		$$\sum_{j=0}^{\infty} \frac{1}{\left(1+\lambda|\xi^1|+\lambda\left|\xi^2-\hat{z}_L^{j}\right|\right)^{\tau}} \leq C$$
		in the last inequality since $\tau>\bar{k}$. This gives that
		\begin{equation}\label{eqs2.48}
			\begin{aligned}
				&\quad \int_{\Omega \backslash B_{1}(\hat{x})}\frac{G(x,\xi)}{|\xi^1|}\Bigl(\sum_{j=0}^{\infty}W_{\hat{x}_L^j,\lambda}\Bigr)^{2^\sharp-2} \sigma(\xi) \sum_{j=0}^{\infty} \frac{\lambda^{\frac{N-2}{2}}}{\left(1+\lambda\left|\xi^1\right|+\lambda\left|\xi^2-\hat{z}_L^j\right|\right)^{\frac{N-2}{2}+\tau}}\mathrm{d}\xi\\
				&\leq C\lambda^{-\frac{2\theta}{N-2}} \int_{\Omega} \frac{G(x,\xi)}{|\xi^1|} \sigma(\xi) \sum_{j=0}^{\infty} \frac{\lambda^{\frac{N}{2}}}{\left(1+\lambda|\xi^1|+\lambda\left|\xi^2-\hat{z}_L^{j}\right|\right)^{\frac{N}{2}+\tau}} \mathrm{d}\xi\\
				& = C\lambda^{-\frac{2\theta}{N-2}} \int_{\mathbb{R}^{N}} \Gamma(x,\xi) \sigma(\xi) \sum_{j=0}^{\infty} \frac{\lambda^{\frac{N}{2}}}{|\xi^1|\left(1+\lambda|\xi^1|+\lambda\left|\xi^2-\hat{z}_L^{j}\right|\right)^{\frac{N}{2}+\tau}} \mathrm{d}\xi\\
				& \leq C\lambda^{-\frac{2\theta}{N-2}} \sigma(x) \sum_{j=0}^{\infty} \frac{\lambda^{\frac{N-2}{2}}}{\left(1+\lambda|y|+\lambda\left|z-\hat{z}_L^{j}\right|\right)^{\frac{N-2}{2}+\tau}}.
			\end{aligned}
		\end{equation}
		Thus, the result is derived from equations \eqref{eqs2.46} and \eqref{eqs2.48}.
	\end{proof}

	Recall that $T$ is defined in \eqref{eq2.15}. We are now prepared to prove the following result.
	\begin{lemma}\label{lm2.5}
		$T$ is a compact bounded linear operator from $\mathbf{E}$ to itself.
	\end{lemma}
	\begin{proof}
		By the definition of $T$, for any $\phi \in \mathbf{E}$, we have
		\begin{equation}\label{eqs2.49}
			\begin{aligned}
				&\quad \|T\phi\|_*=\Bigl\|\left(2^{\sharp}-1\right)(-\Delta)^{-1}\Bigl[\mathbf{P}\Bigl(\frac{M(x)\left(P W_{\hat{x},\lambda}\right)^{2^{\sharp}-2} \phi}{|y|}\Bigr)\Bigr]\Bigr\|_*\\
				& \leq (2^{\sharp}-1)\Bigl\|(-\Delta)^{-1}\Bigl[\frac{M(x)\left(P W_{\hat{x},\lambda}\right)^{2^{\sharp}-2}\phi}{|y|}\Bigr]\Bigr\|_*+(2^{\sharp}-1)\Bigl\|(-\Delta)^{-1}\Bigl[\sum_{i=1}^{h+1}c_{i}\frac{W_{\hat{x},\lambda}^{2^{\sharp}-2}Z_{i}}{|y|}\Bigr]\Bigr\|_*.
			\end{aligned} 		
		\end{equation}
		On the one hand, using Lemma \ref{lm2.4}, we have given the estimate of $\Bigl|(-\Delta)^{-1}\Bigl[\frac{M(x) (P W_{\hat{x},\lambda})^{2^{\sharp}-2} \phi}{|y|}\Bigr]\Bigr|$.
		On the other hand, using Lemma \ref{lm2.3}, we obtain
		\begin{equation}\label{eqs2.50}
			\begin{aligned}
				\Bigl|(-\Delta)^{-1}\Bigl(c_{i}\frac{W_{\hat{x},\lambda}^{2^{\sharp}-2}Z_{i}}{|y|}\Bigr)\Bigr| &=\Bigl|\int_{\Omega} G(x,\xi) c_{i}\frac{W_{\hat{x},\lambda}^{2^{\sharp}-2}Z_{i}}{|\xi^1|}\mathrm{d}\xi\Bigr|\\
				& \leq c_{i}\lambda^{\alpha(i)}\Bigl|\int_{\Omega} G(x,\xi) \frac{W_{\hat{x},\lambda}^{2^{\sharp}-1}}{|\xi^1|}\mathrm{d}\xi\Bigr|\\
				& = c_{i}\lambda^{\alpha(i)}PW_{\hat{x},\lambda} \leq c_{i}\lambda^{\alpha(i)} \sum_{j=0}^{\infty}W_{\hat{x}_L^j,\lambda} \\
				& \leq \frac{C}{\lambda^{\beta}} \|\phi\|_* \sum_{j=0}^{\infty}\frac{\lambda^{\frac{N-2}{2}}}{\left(1+\lambda |y|+\lambda |z-\hat{z}_L^j|\right)^{\frac{N-2}{2}+\tau}}.
			\end{aligned}
		\end{equation}
		By combining \eqref{eqs2.39}, \eqref{eqs2.49} and \eqref{eqs2.50}, we establish that $\|T \phi\|_{*} \leq C\|\phi\|_{*}$. This inequality demonstrates that $T$ is a bounded linear operator from $\mathbf{E}$ to itself.
		
		Now we prove that $T$ is a compact operator. Suppose that $\phi_n$ is a bounded sequence in $\mathbf{E}$. On the one hand, for any $\varepsilon>0$, from the proofs of \eqref{eqs2.39} and \eqref{eqs2.50}, we observe that there exists $R>0$ sufficiently large, such that
		\begin{equation}\label{eqs2.51}
			\Bigl(\sigma(x) \sum_{j=0}^{\infty}\frac{\lambda^{\frac{N-2}{2}}}{\left(1+\lambda |y|+\lambda |z-\hat{z}_L^j|\right)^{\frac{N-2}{2}+\tau}}\Bigr)^{-1}\left|T \phi_n\right|<\varepsilon, \ x \in \Omega \backslash B_{R}(0).
		\end{equation}
		On the other hand, from the equation
		$$
		-\Delta T \phi_n=(2^{\sharp}-1)\Bigl(\frac{M(x)\left(P W_{\hat{x},\lambda}\right)^{2^{\sharp}-2}\phi_n}{|y|}+\sum_{i=1}^{h+1}c_{i}\frac{W_{\hat{x},\lambda}^{2^{\sharp}-2}Z_{i}}{|y|}\Bigr),
		$$ 
		and the definition of $\|\cdot\|_*$, using Lemma \ref{lm2.3} and \ref{lmA.1}, we know that
		$$
		\frac{M(x)\left(P W_{\hat{x},\lambda}\right)^{2^{\sharp}-2}\phi_n}{|y|}, c_{i}\frac{W_{\hat{x},\lambda}^{2^{\sharp}-2}Z_{i}}{|y|} \in L_{\text{loc}}^{s}(\Omega), \ \text{for any $0<s<k$}.
		$$
		Then applying Remark \ref{re1.3}, $L^p$ estimate and the Sobolev embedding theorem, it follows that $T \phi_n$ is bounded in $C_{\text{loc}}^1(\Omega)$. Consequently, $T \phi_n$ has a subsequence, still denoted by $T \phi_n$, which converges to $u_{0}$ uniformly in $B_{R}(0)\cap \Omega$ for any $R>0$. Moreover, since $\left\|T \phi_n\right\|_{*} \leq C<+\infty$, we deduce that $\left\|u_{0}\right\|_{*} \leq C$, which implies $u_{0} \in \mathbf{E}$. By applying \eqref{eqs2.51} to $\phi_n$ and letting $n \rightarrow\infty$, we obtain 
		$$
		\Bigl(\sigma(x) \sum_{j=0}^{\infty}\frac{\lambda^{\frac{N-2}{2}}}{\left(1+\lambda |y|+\lambda |z-\hat{z}_L^j|\right)^{\frac{N-2}{2}+\tau}}\Bigr)^{-1}\left|u_{0}\right|<\varepsilon, \ x \in \Omega \backslash B_{R}(0),
		$$ 
		if $R>0$ is sufficiently large. This result, combined with \eqref{eqs2.51} and $T \phi_n \rightarrow u_{0}$ uniformly in $B_{R}(0) \cap \Omega$, implies that $\left\|T \phi_n-u_{0}\right\|_{*} \rightarrow 0$. Therefore, $T$ is a compact operator.
	\end{proof}
	
	Now we prove that $I-T$ is injective in $\mathbf{E}$. For this purpose, we consider the equation $\phi-T \phi= (-\Delta)^{-1} f$, where $\phi \in \mathbf{E}$ and $f \in \mathbf{F}$. Equivalently, we can reformulate this as the following linear problem:
	\begin{equation}\label{eqs2.52}
		-\Delta \phi-\left(2^{\sharp}-1\right) \frac{M(x)\left(P W_{\hat{x},\lambda}\right)^{2^{\sharp}-2} \phi}{|y|}=f+\sum_{i=1}^{h+1} c_{i} \frac{W_{\hat{x},\lambda}^{2^{\sharp}-2} Z_{i}}{|y|}, \ \phi \in \mathbf{E},
	\end{equation} 
	for some constants $c_{i}$, where $f$ is a function in $\mathbf{F}$.
	
	\begin{lemma}\label{lm2.6}
		If $\phi \in \mathbf{E}$ solves \eqref{eqs2.52}, then $\|\phi\|_{*} \leq C\|f\|_{**}$ for some constant $C>0$, independent of $(\hat{x},\lambda)$. In particular, $I-T$ is injective in $\mathbf{E}$.
	\end{lemma}
	\begin{proof}
		We write
		\begin{equation}\label{eqs2.53}
			\begin{aligned}
				\phi(x)= & \left(2^{\sharp}-1\right) \int_{\Omega} G(x, \xi) \frac{M(\xi)\left(P W_{\hat{x},\lambda}\right)^{2^{\sharp}-2} \phi(\xi)}{|\xi^1|} \mathrm{d}\xi \\
				& +\int_{\Omega} G(x, \xi)\Bigl(f(\xi)+\sum_{i=1}^{h+1} c_{i} \frac{W_{\hat{x},\lambda}^{2^{\sharp}-2} Z_{i}}{|\xi^1|} \Bigr)\mathrm{d}\xi.
			\end{aligned}
		\end{equation}
		On the one hand, the first term in the right hand side of \eqref{eqs2.53} has been estimated in Lemma \ref{lm2.4}, that is 
		\begin{align}\label{eqs2.54} 
			& \quad \left(2^{\sharp}-1\right) \Bigl|\int_{\Omega} G(x, \xi) \frac{M(\xi)\left(P W_{\hat{x},\lambda}\right)^{2^{\sharp}-2} \phi(\xi)}{|\xi^1|} \mathrm{d}\xi\Bigr| \nonumber\\
			& \leq C\|\phi\|_{*} \Bigl[\sigma(x) \sum_{j=0}^{\infty} \frac{\lambda^{\frac{N-2}{2}}}{\left(1+\lambda|y|+\lambda\left|z-\hat{z}_L^{j}\right|\right)^{\frac{N-2}{2}+\tau+\theta}}\\
			& \quad +\lambda^{-\frac{2 \theta}{N-2}} \sigma(x) \sum_{j=0}^{\infty} \frac{\lambda^{\frac{N-2}{2}}}{\left(1+\lambda|y|+\lambda\left|z-\hat{z}_L^{j}\right|\right)^{\frac{N-2}{2}+\tau}}\Bigr]. \nonumber
		\end{align}
		On the other hand, it is easy to see that
		\begin{equation}\label{eqs2.55}
			\begin{aligned}
				& \quad \Bigl|\int_{\Omega} G(x, \xi) f(\xi) \mathrm{d}\xi\Bigr|\\
				& \leq C\|f\|_{**} \int_{\Omega} G(x, \xi) \sigma(\xi) \sum_{j=0}^{\infty} \frac{\lambda^{\frac{N}{2}}}{\left|\xi^1\right|\left(1+\lambda\left|\xi^1\right|+\lambda\left|\xi^2-\hat{z}_L^j\right|\right)^{\frac{N}{2}+\tau}} \mathrm{d}\xi \\
				& =C\|f\|_{**} \int_{\mathbb{R}^{N}} \frac{1}{|x-\xi|^{N-2}} \sigma(\xi) \sum_{j=0}^{\infty} \frac{\lambda^{\frac{N}{2}}}{|\xi^1|\left(1+\lambda\left|\xi^1\right|+\lambda\left|\xi^2-\hat{z}_L^j\right|\right)^{\frac{N}{2}+\tau}} \mathrm{d}\xi \\
				& \leq C\|f\|_{**} \sigma(x) \sum_{j=0}^{\infty} \frac{\lambda^{\frac{N-2}{2}}}{\left(1+\lambda\left|y\right|+\lambda\left|z-\hat{z}_L^j\right|\right)^{\frac{N-2}{2}+\tau}}
			\end{aligned}
		\end{equation} 
		and
		\begin{equation}\label{eqs2.56}
			\begin{aligned}
				& \quad \Bigl|\int_{\Omega} G(x, \xi) \frac{W_{\hat{x},\lambda}^{2^{\sharp}-2} Z_{i}}{|\xi^1|} \mathrm{d}\xi\Bigr| \\
				& \leq C \lambda^{\alpha(i)} \int_{\Omega} G(x, \xi) \frac{W_{\hat{x},\lambda}^{2^{\sharp}-1}}{|\xi^1|} \mathrm{d}\xi \\
				& \leq C \lambda^{\alpha(i)} \sum_{j=0}^{\infty} \frac{\lambda^{\frac{N-2}{2}}}{\left(1+\lambda |y|+\lambda|z-\hat{z}_{L}^{j}|\right)^{\frac{N-2}{2}+\tau}}.
			\end{aligned}
		\end{equation}
		
		It remains to estimate $c_{i}$. We use \eqref{eqs2.52} to find that
		\begin{equation}\label{eqs2.57}
			-\sum_{i=1}^{h+1} c_{i} \int_{\Omega} \frac{W_{\hat{x},\lambda}^{2^{\sharp}-2} Z_{i}}{|\xi^1|} \partial_{m} P W_{\hat{x},\lambda} =\int_{\Omega}\Bigl[\left(2^{\sharp}-1\right) \frac{M(\xi) \left(P W_{\hat{x},\lambda}\right)^{2^{\sharp}-2} \phi}{|\xi^1|}+f\Bigr] \partial_{m} P W_{\hat{x},\lambda}.
		\end{equation}
		We have
		\begin{equation}\label{eqs2.58}
			\Bigl|\int_{\Omega} f \partial_{m} P W_{\hat{x},\lambda} \Bigr| \leq C\|f\|_{**} \lambda^{\alpha(m)} \int_{\Omega} \sigma(\xi) \sum_{j=0}^{\infty} \frac{\lambda^{\frac{N}{2}}}{\left|\xi^1\right|\left(1+\lambda\left|\xi^1\right|+\lambda\left|\xi^2-\hat{z}_L^j\right|\right)^{\frac{N}{2}+\tau}} \sum_{j=0}^{\infty} W_{\hat{x}_L^j, \lambda}.
		\end{equation}
		Similar to the proofs of \eqref{eq2.23} and \eqref{eq2.24}, we can prove that
		\begin{align}\label{eqs2.59} 
			& \quad \int_{B_1(\hat{x})} \sigma(\xi) \sum_{j=0}^{\infty} \frac{\lambda^{\frac{N}{2}}}{\left|\xi^1\right|\left(1+\lambda\left|\xi^1\right|+\lambda\left|\xi^2-\hat{z}_L^j\right|\right)^{\frac{N}{2}+\tau}} \sum_{j=0}^{\infty} W_{\hat{x}_L^j, \lambda} \nonumber\\
			& \leq C \lambda^{N-1} \int_{B_1(\hat{x})} \frac{\sigma(\xi)}{\left|\xi^1\right|\left(1+\lambda\left|\xi^1\right|+\lambda\left|\xi^2-\hat{z}\right|\right)^{N-2+\frac{N}{2}+\tau}} \\
			& \leq C \lambda^{N-1} \int_{B_1(\hat{x})}\left(\frac{1+\lambda\left|\xi^1\right|+\lambda\left|\xi^2-\hat{z}\right|}{\lambda}\right)^\tau \frac{1}{\left|\xi^1\right|\left(1+\lambda\left|\xi^1\right|+\lambda\left|\xi^2-\hat{z}\right|\right)^{N-2+\frac{N}{2}+\tau}} \nonumber\\ 
			& \leq C \frac{\lambda^{N-1}}{\lambda^\tau} \int_{B_{\lambda}(0)} \frac{1}{\frac{|{\xi}^1|}{\lambda}\left(1+|{\xi}^1|+|{\xi}^2|\right)^{\frac{3N}{2}-2}} \frac{1}{\lambda^{N}}\leq \frac{C}{\lambda^\tau}. \nonumber
		\end{align}
		Moreover, we have
		\begin{equation}\label{eqs2.60}
			\begin{aligned}
				& \quad \int_{\Omega \backslash B_1(\hat{x})} \sigma(\xi) \frac{\lambda^{\frac{N}{2}}}{\left|\xi^1\right|\left(1+\lambda\left|\xi^1\right|+\lambda\left|\xi^2-\hat{z}\right|\right)^{\frac{N}{2}+\tau}} W_{\hat{x},\lambda} \\
				& \leq \int_{\mathbb{R}^{N} \backslash B_{\lambda}(0)} \frac{1}{\left|\xi^1\right|\left(1+\left|\xi^1\right|+\left|\xi^2\right|\right)^{N-2+\frac{N}{2}+\tau}} \leq \frac{C}{\lambda^{N-2-\theta}}.
			\end{aligned}
		\end{equation}
		For $\xi \in \Omega$, it holds that
		\begin{equation*}
			\begin{aligned}
				\sum_{j=1}^{\infty} W_{\hat{x}_L^j, \lambda} & \leq \frac{C \lambda^{\frac{N-2}{2}}}{\left(1+\lambda\left|\xi^1\right|+\lambda\left|\xi^2-\hat{z}\right|\right)^{2\theta}} \sum_{j=1}^{\infty} \frac{1}{(1+\lambda\left|\xi^1\right|+\lambda\left|\xi^2-\hat{z}_L^j\right|)^{N-2-2 \theta}} \\
				& \leq \frac{C \lambda^{\frac{N-2}{2}}}{\left(1+\lambda\left|\xi^1\right|+\lambda\left|\xi^2-\hat{z}\right|\right)^{2\theta}} \frac{1}{(\lambda L)^{N-2-2 \theta}}
			\end{aligned}
		\end{equation*}
		and
		\begin{equation*}
			\sum_{j=1}^{\infty} \frac{\lambda^{\frac{N}{2}}}{\left|\xi^1\right|\left(1+\lambda\left|\xi^1\right|+\lambda\left|\xi^2-\hat{z}_L^j\right|\right)^{\frac{N}{2}+\tau}} \leq \frac{C \lambda^{\frac{N}{2}}}{|y| \left(1+\lambda\left|\xi^1\right|+\lambda\left|\xi^2-\hat{z}\right|\right)^{2+\theta}} \frac{1}{(\lambda L)^{\frac{N}{2}+\tau-2-\theta}}.
		\end{equation*}
		Thus, we obtain
		\begin{equation}\label{eqs2.61}
			\begin{aligned}
				& \quad \int_{\Omega \backslash B_1(\hat{x})} \sigma(\xi) \frac{\lambda^{\frac{N}{2}}}{\left|\xi^1\right|\left(1+\lambda\left|\xi^1\right|+\lambda\left|\xi^2-\hat{z}\right|\right)^{\frac{N}{2}+\tau}} \sum_{j=1}^{\infty} W_{\hat{x}_L^j, \lambda} \\
				& \leq \int_{\mathbb{R}^{N} \backslash B_{\lambda}(0)} \frac{1}{|\xi^1| (1+|\xi^1|+|\xi^2|)^{2 \theta+\frac{N}{2}+\tau}} \frac{C}{(\lambda L)^{N-2-2 \theta}} \leq \frac{C}{(\lambda L)^{N-2-2 \theta}}
			\end{aligned}
		\end{equation}
		and
		\begin{equation}\label{eqs2.62}
			\begin{aligned}
				& \quad \int_{\Omega \backslash B_1(\hat{x})} \sigma(\xi) \sum_{j=1}^{\infty} \frac{\lambda^{\frac{N}{2}}}{\left|\xi^1\right|\left(1+\lambda\left|\xi^1\right|+\lambda\left|\xi^2-\hat{z}_L^j\right|\right)^{\frac{N}{2}+\tau}} W_{\hat{x},\lambda} \\
				& \leq \int_{\mathbb{R}^{N} \backslash B_{\lambda}(0)} \frac{1}{\left|\xi^1\right| (1+\left|\xi^1\right|+\left|\xi^2\right|)^{N+\theta}} \frac{C}{(\lambda L)^{\frac{N}{2}+\tau-2-\theta}} \leq \frac{C}{\lambda ^{N-2-\theta}}.
			\end{aligned}
		\end{equation}
		Therefore, there holds
		\begin{align}\label{eqs2.63} 
			& \quad \int_{\Omega \backslash B_1(\hat{x})} \sigma(\xi) \sum_{j=1}^{\infty} \frac{\lambda^{\frac{N}{2}}}{\left|\xi^1\right|\left(1+\lambda\left|\xi^1\right|+\lambda\left|\xi^2-\hat{z}_L^j\right|\right)^{\frac{N}{2}+\tau}} \sum_{j=1}^{\infty} W_{\hat{x}_L^j, \lambda} \nonumber\\
			& \leq C \lambda^{N-1} \int_{\Omega \backslash B_1(\hat{x})} \Bigl(\frac{1}{\left|\xi^1\right| \left(1+\lambda\left|\xi^1\right|+\lambda\left|\xi^2-\hat{z}\right|\right)^{\frac{N}{2}}} \frac{1}{(\lambda L)^{\tau}} \\
			&\quad \quad \quad \quad \quad \quad \quad \quad \frac{1}{ \left(1+\lambda\left|\xi^1\right|+\lambda\left|\xi^2-\hat{z}\right|\right)^{\frac{N-2}{2}+\theta}} \frac{1}{(\lambda L)^{\frac{N-2}{2}-\theta}} \Bigr) \nonumber\\
			& \leq C \int_{\mathbb{R}^{N} \backslash B_{\lambda}(0)} \frac{1}{\left|\xi^1\right| (1+\left|\xi^1\right|+\left|\xi^2\right|)^{N-1+\theta}} \frac{1}{(\lambda L)^{N-2-2 \theta}} \leq \frac{C}{(\lambda L)^{N-2-2 \theta}}. \nonumber
		\end{align}
		Combining \eqref{eqs2.60}-\eqref{eqs2.63}, we obtain
		\begin{equation}\label{eqs2.64}
			\int_{\Omega \backslash B_1(\hat{x})} \sigma(\xi) \sum_{j=0}^{\infty} \frac{\lambda^{\frac{N}{2}}}{\left|\xi^1\right|\left(1+\lambda\left|\xi^1\right|+\lambda\left|\xi^2-\hat{z}_L^j\right|\right)^{\frac{N}{2}+\tau}} \sum_{j=0}^{\infty} W_{\hat{x}_L^j, \lambda} \leq \frac{C}{\lambda^{N-2-\theta}}.
		\end{equation}
		It follows from \eqref{eqs2.58}, \eqref{eqs2.59} and \eqref{eqs2.64} that
		\begin{equation}\label{eqs2.65}
			\Bigl|\int_{\Omega} f \partial_{m} P W_{\hat{x},\lambda}\Bigr| \leq C \frac{\|f\|_{**} \lambda^{\alpha(m)}}{\lambda^\tau}.
		\end{equation}
		Using \eqref{eqs2.57}, \eqref{eq2.18}, \eqref{eqs2.65} and \eqref{eqs2.38}, we see
		\begin{equation*}\label{eqs2.66}
			\left|c_{i}\right| \leq C\left(\|f\|_{**}+o(1)\|\phi\|_{*}\right) \frac{\lambda^{-\alpha(i)}}{\lambda^\tau},
		\end{equation*} 
		which with \eqref{eqs2.56} gives that
		\begin{equation}\label{eqs2.67}
			\begin{aligned}
				& \quad \Bigl| \int_{\Omega} G(x, \xi) \sum_{i=1}^{h+1} c_{i} \frac{W_{\hat{x},\lambda}^{2^{\sharp}-2} Z_{i}}{|\xi^1|} \mathrm{d}\xi\Bigr| \\
				& \leq C\left(\|f\|_{**}+o(1)\|\phi\|_{*}\right) \frac{1}{\lambda^\tau} \sum_{j=0}^{\infty} \frac{\lambda^{\frac{N-2}{2}}}{\left(1+\lambda |y|+\lambda|z-\hat{z}_{L}^{j}|\right)^{\frac{N-2}{2}+\tau}} \\
				& \leq C\left(\|f\|_{**}+o(1)\|\phi\|_{*}\right) \sigma(x) \sum_{j=0}^{\infty} \frac{\lambda^{\frac{N-2}{2}}}{\left(1+\lambda |y|+\lambda|z-\hat{z}_{L}^{j}|\right)^{\frac{N-2}{2}+\tau}}.
			\end{aligned}
		\end{equation}
		Combining \eqref{eqs2.53}, \eqref{eqs2.54}, \eqref{eqs2.55} and \eqref{eqs2.67}, we are led to
		\begin{equation}\label{eqs2.68}
			\begin{aligned}
				&\quad |\phi(x)|\Bigl(\sigma(x) \sum_{j=0}^{\infty} \frac{\lambda^{\frac{N-2}{2}}}{\left(1+\lambda |y|+\lambda|z-\hat{z}_{L}^{j}|\right)^{\frac{N-2}{2}+\tau}}\Bigr)^{-1} \\
				& \leq C\Bigl(\|f\|_{**}+o(1)\|\phi\|_{*}+\frac{\sum_{j=0}^{\infty} \frac{1}{\left(1+\lambda |y|+\lambda|z-\hat{z}_{L}^{j}|\right)^{\frac{N-2}{2}+\tau+\theta}}}{\sum_{j=0}^{\infty} \frac{1}{\left(1+\lambda |y|+\lambda|z-\hat{z}_{L}^{j}|\right)^{\frac{N-2}{2}+\tau}}}\|\phi\|_{*}\Bigr).
			\end{aligned}
		\end{equation} 
		Thus we have proved that $\|\phi\|_{*} \leq C\|f\|_{**}$.
		
		Next we prove that $I-T$ is injective in $\mathbf{E}$. Suppose that there exist sequences $L_n \rightarrow+\infty$, $\hat{x}^{n} \rightarrow 0$, $\lambda_n \rightarrow+\infty$, $f_n \in \mathbf{F}$ and $\phi_n \in \mathbf{E}$, such that $\phi_n$ solves \eqref{eqs2.52} with $f=f_n$, $c_i=c_i^n$, $\left\|\phi_n\right\|_{*}=1$ and $\left\|f_n\right\|_{**} \rightarrow 0$ as $n \rightarrow\infty$. Denote that $\hat{x}^{n,j}_{L_n}=(0,\hat{z}^{n,j}_{L_n})\in \mathbb{R}^k \times \mathbb{R}^{h}$, where $\hat{z}^{n,j}_{L_n}=\hat{z}^{n}-L_n\tilde{P}^{j}$.
		
		From \eqref{eqs2.68}, we see that the maximum of $|\phi_n (x)|\Bigl(\sigma(x) \sum_{j=0}^{\infty} \frac{\lambda_n^{\frac{N-2}{2}}}{\left(1+\lambda_n|y|+\lambda_n|z-\hat{z}_{L_n}^{n,j}|\right)^{\frac{N-2}{2}+\tau}}\Bigr)^{-1}$ in $\Omega$ can only be achieved in $B_{R \lambda_{n}^{-1}}(0)$ for some sufficiently large $R>0$.
		Now we consider a sequence $\tilde{\phi}_n(x):=\lambda_n^{-\frac{N-2}{2}} \phi_n\left(\lambda_n^{-1} x+\hat{x}^{n}\right)$. 
		With the estimate of $c_i^n$ in Lemma \ref{lm2.3}, we may apply a combination of Remark \ref{re1.3}, the $L^p$ estimate and the Sobolev embedding theorem to conclude that as $n \to \infty$, the sequence $\tilde{\phi}_n$ converges in $C^1_{\text{loc}}(\mathbb{R}^N)$ 
		to $\phi_{0}$, which satisfies 
		$$
		-\Delta \phi_{0}-\left(2^{\sharp}-1\right) \frac{W_{0,1}^{2^{\sharp}-2} \phi_{0}}{|y|}=0 \,\,\, \text {in} \,\,\, \mathbb{R}^{N}.
		$$
		Since $\phi_n \in \mathbf{E}$, we conclude that $\phi_{0}=0$. This is a contradiction to $\left\|\phi_n\right\|_{*}=1$.
	\end{proof} 
	
	Now, we will prove Proposition \ref{pro2.2}.
	\begin{proof}[\textbf{Proof of Proposition \ref{pro2.2}.}]
		Since $T$ is compact and $I-T$ is injective, Fredholm theorem implies that $I-T$ is surjective.
	\end{proof}

	\section{Existence of periodic solutions} \label{sec3}
	In this section, we will prove Theorem \ref{th1.1}. We first solve the following problem
	\begin{equation}\label{eq3.1}
		-\Delta\left(P W_{\hat{x},\lambda}+\omega\right)-\frac{M(x)\left(P W_{\hat{x},\lambda}+\omega\right)_{+}^{2^{\sharp}-1}}{|y|}=\sum_{i=1}^{h+1} c_{i} \frac{W_{\hat{x},\lambda}^{2^{\sharp}-2} Z_{i}}{|y|},
	\end{equation}
	for some constants $c_{i}$. To achieve this purpose, we initially give the estimates of $N(\omega)$ and $l$.
	
	\begin{lemma}\label{lm3.2}
		If $N \geq 5$, then
		$$
		\|N(\omega)\|_{**} \leq C\|\omega\|_{*}^{2^{\sharp}-1}, \ \text{for any $\omega \in \mathbf{E}$}.
		$$
	\end{lemma}
	\begin{proof}
		Since $\tau>\bar{k}$, we can check that for $x \in \Omega$, 
		$$
		\sum_{j=0}^{\infty}\frac{1}{\left(1+\lambda\left|y\right|+\lambda\left|z-\hat{z}_L^j\right|\right)^{\tau}} \leq C.
		$$
		Using the H\"older inequality, we have
		\begin{equation*}
			\begin{aligned}
				|N(\omega)| & \leq \frac{C|\omega|^{2^{\sharp}-1}}{|y|} \leq \frac{C\|\omega\|_*^{2^{\sharp}-1}\sigma(x)}{|y|}\Bigl(\sum_{j=0}^{\infty} \frac{\lambda^{\frac{N-2}{2}}}{\left(1+\lambda\left|y\right|+\lambda\left|z-\hat{z}_L^j\right|\right)^{\frac{N-2}{2}+\tau}}\Bigr)^{2^{\sharp}-1} \\
				& \leq \frac{C\|\omega\|_*^{2^{\sharp}-1}\sigma(x)}{|y|}\sum_{j=0}^{\infty} \frac{\lambda^{\frac{N}{2}}}{\left(1+\lambda|y|+\lambda\left|z-\hat{z}_L^{j}\right|\right)^{\frac{N}{2}+\tau}} \Bigl(\sum_{j=0}^{\infty} \frac{1}{\left(1+\lambda|y|+\lambda\left|z-\hat{z}_L^{j}\right|\right)^{\tau}}\Bigr)^{\frac{2}{N-2}} \\
				& \leq \frac{C\|\omega\|_*^{2^{\sharp}-1}\sigma(x)}{|y|}\sum_{j=0}^{\infty} \frac{\lambda^{\frac{N}{2}}}{\left(1+\lambda\left|y\right|+\lambda\left|z-\hat{z}_L^j\right|\right)^{\frac{N}{2}+\tau}}.
			\end{aligned}
		\end{equation*}
		Hence, the result holds.
	\end{proof}

	\begin{lemma}\label{lm3.3}
		If $N \geq 5$, then
		$$
		\|l\|_{**} \leq \frac{C}{\lambda^{1+\theta}}.
		$$
	\end{lemma}
	\begin{proof}
		We have
		$$
		l=\frac{M(x)\Bigl(\left(P W_{\hat{x},\lambda}\right)^{2^{\sharp}-1}-W_{\hat{x},\lambda}^{2^{\sharp}-1}\Bigr)}{|y|} + \frac{(M(x)-1) W_{\hat{x},\lambda}^{2^{\sharp}-1}}{|y|}:=J_1+J_2.
		$$ 	
		Since $2^{\sharp} > 1$, it holds that
		\begin{equation*}\label{eq3.5}
			\left|J_1\right| \leq C \frac{W_{\hat{x},\lambda}^{2^{\sharp}-2}\left|\varphi_{\hat{x},\lambda}\right|}{|y|} + C\frac{\left|\varphi_{\hat{x},\lambda}\right|^{2^{\sharp}-1}}{|y|}:=J_{11}+J_{12}.
		\end{equation*}
		To estimate $J_{11}$, for $x \in B_1(\hat{x})$, using Lemma \ref{lmA.2} and $\beta>N-2$, we have
		\begin{equation}\label{eq3.6}
			\begin{aligned}
				\left|J_{11}\right| & \leq \frac{C \lambda}{|y| \left(1+\lambda\left|y\right|+\lambda\left|z-\hat{z}\right|\right)^2} \frac{1}{\lambda^{\frac{N-2}{2}} L^{N-2}} \leq \frac{C \lambda^{\frac{N}{2}}}{|y| \left(1+\lambda\left|y\right|+\lambda\left|z-\hat{z}\right|\right)^2} \frac{1}{\lambda^{\beta}} \\
				& \leq \frac{C \lambda^{\frac{N}{2}}}{|y| \left(1+\lambda\left|y\right|+\lambda\left|z-\hat{z}\right|\right)^{\frac{N}{2}+\tau}} \sigma(x) \frac{1}{\lambda^{\beta-\tau-\frac{N}{2}+2}} \\
				& \leq \frac{C \sigma(x) \lambda^{\frac{N}{2}}}{|y| \left(1+\lambda\left|y\right|+\lambda\left|z-\hat{z}\right|\right)^{\frac{N}{2}+\tau}} \frac{1}{\lambda^{1+\theta}}.
			\end{aligned}
		\end{equation}
		For $x \in \Omega \backslash B_1(\hat{x})$, we have
		\begin{equation}\label{eq3.7}
			\begin{aligned}
				\left|J_{11}\right| & \leq C \frac{W_{\hat{x},\lambda}^{2^{\sharp}-2} \sum_{j=0}^{\infty} W_{\hat{x}_L^j, \lambda}}{|y|}\\
				& \leq C \frac{\lambda^{\frac{N}{2}}}{\left(1+\lambda |y|+\lambda |z-\hat{z}|\right)^2} \sum_{j=0}^{\infty} \frac{1}{|y| \left(1+\lambda |y|+\lambda |z-\hat{z}_L^j|\right)^{N-2}} \\
				& \leq \frac{C \lambda^{\frac{N}{2}}}{|y| \left(1+\lambda |y|+\lambda |z-\hat{z}|\right)^{\frac{N}{2}+\tau}} \\
				&\quad \quad \Bigl(\frac{1}{\left(1+\lambda |y|+\lambda |z-\hat{z}|\right)^{\frac{N}{2}-\tau}} + \sum_{j=1}^{\infty} \frac{1}{\left(1+\lambda |y|+\lambda |z-\hat{z}_L^j|\right)^{\frac{N}{2}-\tau}}\Bigr)\\
				& \leq \frac{C \lambda^{\frac{N}{2}}}{|y| \left(1+\lambda |y|+\lambda |z-\hat{z}|\right)^{\frac{N}{2}+\tau}} \Bigl(\frac{1}{\lambda^{\frac{N}{2}-\tau}} + \frac{1}{(\lambda L)^{\frac{N}{2}-\tau}}\Bigr)\\
				& \leq C \sigma(x) \sum_{j=0}^{\infty}\frac{\lambda^{\frac{N}{2}}}{|y| \left(1+\lambda |y|+\lambda |z-\hat{z}_L^j|\right)^{\frac{N}{2}+\tau}} \frac{1}{\lambda^{1+\theta}}.
			\end{aligned}
		\end{equation}
		By \eqref{eq3.6} and \eqref{eq3.7}, we obtain
		\begin{equation}\label{eq3.8}
			\|J_{11}\|_{**}\leq \frac{C}{\lambda^{1+\theta}}.
		\end{equation}
		By the same argument as that of the proof of \eqref{eq3.6}, for $x \in B_1(\hat{x})$, it holds that
		\begin{equation}\label{eq3.9}
			\begin{aligned}
				& \left|J_{12}\right| \leq \frac{C}{|y|} \Bigl(\frac{1}{\lambda^{\frac{N-2}{2}} L^{N-2}}\Bigr)^{2^{\sharp}-1}\leq \frac{C \lambda^{\frac{N}{2}}}{|y| \lambda^{\left(2^{\sharp}-1\right) \beta}} \leq \frac{C \sigma(x) \lambda^{\frac{N}{2}}}{|y| \left(1+\lambda\left|y\right|+\lambda\left|z-\hat{z}\right|\right)^{\frac{N}{2}+\tau}} \frac{1}{\lambda^{1+\theta}}.
			\end{aligned}
		\end{equation}
		Similar to the proof of \eqref{eq3.7}, by H\"older inequality, for $x \in \Omega \backslash B_1(\hat{x})$, we obtain
		\begin{align}\label{eq3.10} 
			|J_{12}| & \leq C\frac{\sum_{j=0}^{\infty} W_{\hat{x}_L^j, \lambda}^{2^{\sharp}-1}}{|y|} \leq \frac{CW_{\hat{x},\lambda}^{2^{\sharp}-1}}{|y|} + C\frac{\sum_{j=1}^{\infty} W_{\hat{x}_L^j, \lambda}^{2^{\sharp}-1}}{|y|} \nonumber\\
			& \leq \frac{C\lambda^{\frac{N}{2}}}{|y|(1+\lambda|y|+\lambda |z-\hat{z}|)^N} + \frac{C}{|y|}\Bigl(\sum_{j=1}^{\infty} \frac{\lambda^{\frac{N-2}{2}}}{(1+\lambda|y|+\lambda |z-\hat{z}_L^j|)^{N-2}}\Bigr)^{\frac{N}{N-2}} \nonumber\\
			& \leq \frac{C\lambda^{\frac{N}{2}}}{|y| (1+\lambda|y|+\lambda |z-\hat{z}|)^{\frac{N}{2}+\tau}}\frac{1}{\lambda^{1+\theta}} \\
			& \quad + C \sum_{j=1}^{\infty} \frac{\lambda^{\frac{N}{2}}}{|y| (1+\lambda|y|+\lambda |z-\hat{z}_L^j|)^{\frac{N}{2}+\tau}} \Bigl(\sum_{j=1}^{\infty} \frac{1}{(1+\lambda|y|+\lambda |z-\hat{z}_L^j|)^{\tau}}\Bigr)^{\frac{2}{N-2}} \nonumber\\
			& \leq C\sigma(x) \sum_{j=0}^{\infty} \frac{\lambda^{\frac{N}{2}}}{|y| (1+\lambda|y|+\lambda |z-\hat{z}_L^j|)^{\frac{N}{2}+\tau}} \Bigl(\frac{1}{\lambda^{1+\theta}}+\frac{1}{\left(\lambda L\right)^{\frac{2\tau}{N-2}}}\Bigr) \nonumber\\
			& \leq C\sigma(x) \sum_{j=0}^{\infty} \frac{\lambda^{\frac{N}{2}}}{|y| (1+\lambda|y|+\lambda |z-\hat{z}_L^j|)^{\frac{N}{2}+\tau}}\frac{1}{\lambda^{1+\theta}}, \nonumber
		\end{align} 
		where in the last inequality, we use the assumption \eqref{aeq2.4}.
		Combining \eqref{eq3.9} and \eqref{eq3.10}, we obtain
		\begin{equation}\label{eq3.11}
			\|J_{12}\|_{**}\leq \frac{C}{\lambda^{1+\theta}}.
		\end{equation}
		Thus, \eqref{eq3.8} and \eqref{eq3.11} yield that
		\begin{equation}\label{eq3.12}
			\|J_1\|_{**}\leq \frac{C}{\lambda^{1+\theta}}.
		\end{equation}
		Now, we estimate $J_2$. 
		When $x \in B_1(\hat{x})$, we see
		\begin{equation}\label{eq3.14}
			\begin{aligned}
				|J_2| &\leq C\frac{|x|^{\beta} W_{\hat{x},\lambda}^{2^{\sharp}-1}}{|y|} 
				\leq C\frac{\lambda^\frac{N}{2}}{|y|(1+\lambda|y|+\lambda|z-\hat{z}|)^{\frac{N}{2}+\tau}} \frac{|\hat{x}|^{\beta}+|x-\hat{x}|^{\beta}}{(1+\lambda|y|+\lambda|z-\hat{z}|)^{\frac{N}{2}-\tau}} \\
				& \leq C\sigma(x) \frac{\lambda^\frac{N}{2}}{|y| (1+\lambda|y|+\lambda|z-\hat{z}|)^{\frac{N}{2}+\tau}} \Bigl(\lambda^{\tau-1}|\hat{x}|^{\beta}+\frac{1}{\lambda^{\frac{N}{2}-\tau}}\Bigr) \\ 
				& \leq	C\sigma(x) \sum_{j=0}^{\infty} \frac{\lambda^\frac{N}{2}}{|y|(1+\lambda|y|+\lambda|z-\hat{z}_L^j|)^{\frac{N}{2}+\tau}} \Bigl(\lambda^{\tau-1}|\hat{x}|^{\beta} + \frac{1}{\lambda^{1+\theta}}\Bigr).	 		
			\end{aligned}
		\end{equation} 
		When $x \in \Omega \backslash B_1(\hat{x})$, it follows that
		\begin{equation}\label{eq3.13}
			|J_2|\leq C\sigma(x) \frac{ \lambda^{\frac{N}{2}}}{|y|(1+\lambda|y|+\lambda|z-\hat{z}|)^{\frac{N}{2}+\tau}} \frac{1}{\lambda^{\frac{N}{2}-\tau}}.
		\end{equation}
		Hence, by \eqref{eq3.14}, \eqref{eq3.13} and \eqref{aeq2.4}, we have proved that
		\begin{equation}\label{eq3.15}
			\left\|J_2\right\|_{**} \leq \frac{C}{\lambda^{1+\theta}}+C \lambda^{\tau-1}|\hat{x}|^{\beta} \leq \frac{C}{\lambda^{1+\theta}}. 
		\end{equation} 
		Combining \eqref{eq3.12} with \eqref{eq3.15}, Lemma \ref{lm3.3} is proved.
	\end{proof}

	Next, we will combine Lemmas \ref{lm3.2} and \ref{lm3.3} with the contraction mapping theorem to prove the following result.
	\begin{proposition}\label{pro3.1}
		Under the assumptions of Theorem \ref{th1.1}, if $L>0$ is sufficiently large, then \eqref{eq3.1} admits a unique solution $\omega=\omega(\hat{x},\lambda)$ in $\mathbf{E}$, such that
		$$
		\|\omega\|_{*} \leq \frac{C}{\lambda^{1+\theta}}.
		$$
	\end{proposition}
	\begin{proof}
		In view of Proposition \ref{pro2.2}, \eqref{eq3.1} is equivalent to
		\begin{equation}\label{eq3.2}
			\omega=B \omega:=(I-T)^{-1}\left[(-\Delta)^{-1} \mathbf{P}(N(\omega))\right]+(I-T)^{-1}\left[(-\Delta)^{-1} \mathbf{P} l\right], \ \omega \in \mathbf{E},
		\end{equation}
		where
		\begin{equation*}\label{eq3.3}
			l=\frac{M(x) (PW_{\hat{x},\lambda})^{2^{\sharp}-1}-W_{\hat{x},\lambda}^{2^{\sharp}-1}}{|y|}
		\end{equation*}
		and
		\begin{equation*}\label{eq3.4}
			N\left(\omega\right)=\frac{M(x)\Bigl[\left(PW_{\hat{x}, \lambda}+\omega\right)_{+}^{2^{\sharp}-1}-(PW_{\hat{x},\lambda})^{2^{\sharp}-1}-\left(2^{\sharp}-1\right) (PW_{\hat{x},\lambda})^{2^{\sharp}-2} \omega\Bigr]}{|y|}.
		\end{equation*}
		By Lemmas \ref{lm2.1} and \ref{lm2.6}, we have
		$$
		\|B \omega\|_{*} \leq C\|l\|_{**}+C\|N(\omega)\|_{**}, \ \text{for any $\omega \in \mathbf{E}$} 
		$$
		and
		$$
		\left\|B\omega_1-B\omega_2\right\|_{*} \leq C\left\|N\left(\omega_1\right)-N\left(\omega_2\right)\right\|_{**}, \ \text{for any $\omega_1,\omega_2 \in \mathbf{E}$.}
		$$
		By Lemmas \ref{lm3.2} and \ref{lm3.3}, we can prove that $B$ is a contraction map from $\{\omega \in \mathbf{E}:\|\omega\|_{*} \leq \frac{M}{\lambda^{1+\theta}}\}$ to itself, where $M>0$ is a sufficiently large constant. So the contraction mapping theorem implies that for sufficiently large $L>0$, \eqref{eq3.2} has a solution $\omega \in \mathbf{E}$, satisfying
		$$
		\|\omega\|_{*} \leq C\|l\|_{**}.
		$$
		Using Lemma \ref{lm3.3}, we obtain the estimate for $\|\omega\|_{*}$.
	\end{proof}

	Now we are in a position to prove Theorem \ref{th1.1}.
	\begin{proof}[\textbf{Proof of Theorem \ref{th1.1}.}]
		We need to show the existence of $(\hat{x},\lambda)=\left(\hat{x}_{0,L}, \lambda_{L}\right)$, such that
		\begin{equation}\label{eq3.16}
			\begin{aligned}
				-\int_{\Omega} \Delta\left(P W_{\hat{x},\lambda}+\omega\right) \partial_{i}\left(P W_{\hat{x},\lambda}\right) 
				&-\int_{\Omega} \frac{M(\xi)\left(P W_{\hat{x},\lambda}+\omega\right)_{+}^{2^{\sharp}-1} \partial_{i}\left(P W_{\hat{x},\lambda}\right)}{\left|\xi^1\right|}=0
			\end{aligned}
		\end{equation}
		for $i=1, \cdots, N+1$. 
		In fact, we have
		\begin{equation}\label{eq3.17}
			\begin{aligned}
				&\quad -\int_{\Omega} \Delta\left(P W_{\hat{x},\lambda}+\omega\right) \partial_{i}\left(P W_{\hat{x},\lambda}\right)-\int_{\Omega} \frac{M(\xi)\left(P W_{\hat{x},\lambda}+\omega\right)_{+}^{2^{\sharp}-1} \partial_{i}\left(P W_{\hat{x},\lambda}\right)}{\left|\xi^1\right|} \\
				& = -\int_{\Omega} \Delta P W_{\hat{x},\lambda}\partial_{i}\left(P W_{\hat{x},\lambda}\right) -\int_{\Omega} \Delta \omega \partial_{i}\left(P W_{\hat{x},\lambda}\right)\\
				& \quad -\int_{\Omega} \frac{M(\xi)\Bigl(\left(P W_{\hat{x},\lambda}+\omega\right)_{+}^{2^{\sharp}-1}-(PW_{\hat{x},\lambda})^{2^{\sharp}-1}-(2^{\sharp}-1)(PW_{\hat{x},\lambda})^{2^{\sharp}-2}\omega\Bigr) \partial_{i}\left(P W_{\hat{x},\lambda}\right)}{\left|\xi^1\right|} \\
				& \quad -\int_{\Omega}\frac{M(\xi)\left(P W_{\hat{x},\lambda}\right)^{2^{\sharp}-1} \partial_{i}\left(P W_{\hat{x},\lambda}\right)}{\left|\xi^1\right|} - (2^{\sharp}-1)\int_{\Omega}\frac{M(\xi)\left(P W_{\hat{x},\lambda}\right)^{2^{\sharp}-2} \partial_{i}\left(P W_{\hat{x},\lambda}\right)\omega}{\left|\xi^1\right|}.
			\end{aligned}
		\end{equation}
		From $\omega \in \mathbf{E}$, we obtain
		\begin{equation}\label{eq3.18}
			-\int_{\Omega} \Delta \omega \partial_{i}\left(P W_{\hat{x},\lambda}\right)=-\int_{\Omega} \Delta\left(\partial_{i}\left(P W_{\hat{x},\lambda}\right)\right) \omega=\left(2^{\sharp}-1\right) \int_{\Omega} \frac{\omega W_{\hat{x},\lambda}^{2^{\sharp}-2} Z_{i}}{\left|\xi^1\right|}=0.
		\end{equation}
		Noting that for any $\gamma>1$, we have $(1+t)_{+}^{\gamma}-1-\gamma t=O\left(t^2\right)$ for all $t \in \mathbb{R}$ if $\gamma \leq 2$. So, we can deduce that when $N \geq 5$,
		\begin{equation}\label{eq3.19}
			\begin{aligned}
				& \quad \Bigl|\int_{\Omega} \frac{M(\xi)\Bigl(\left(P W_{\hat{x},\lambda}+\omega\right)_{+}^{2^{\sharp}-1}-(PW_{\hat{x},\lambda})^{2^{\sharp}-1}-(2^{\sharp}-1)(PW_{\hat{x},\lambda})^{2^{\sharp}-2}\omega\Bigr) \partial_{i}\left(P W_{\hat{x},\lambda}\right)}{\left|\xi^1\right|}\Bigr|\\
				& \leq C\int_{\Omega} \frac{(PW_{\hat{x},\lambda})^{2^{\sharp}-3}|\omega|^2|\partial_{i}(P W_{\hat{x},\lambda})|}{\left|\xi^1\right|} \leq C\lambda^{\alpha(i)}\int_{\Omega} \frac{(PW_{\hat{x},\lambda})^{2^{\sharp}-2}|\omega|^2}{\left|\xi^1\right|}.
			\end{aligned}
		\end{equation}
		On the one hand, similar to the proof of \eqref{eq2.25}, we have
		\begin{equation}\label{eq3.20}
			\begin{aligned}
				& \quad \int_{B_1(\hat{x})} \frac{(PW_{\hat{x},\lambda})^{2^{\sharp}-2}|\omega|^2}{\left|\xi^1\right|} \\
				& \leq C\|\omega\|_{*}^2 \int_{B_1(\hat{x})} \frac{(PW_{\hat{x},\lambda})^{2^{\sharp}-2}}{\left|\xi^1\right|}\Bigl(\frac{(1+\lambda\left|\xi^1\right|+\lambda|\xi^2-\hat{z}|)^\tau}{\lambda^\tau} \sum_{j=0}^{\infty} \frac{\lambda^{\frac{N-2}{2}}}{\left(1+\lambda\left|\xi^1\right|+\lambda|\xi^2-\hat{z}_L^j|\right)^{\frac{N-2}{2}+\tau}}\Bigr)^2 \\
				& \leq C\|\omega\|_{*}^2 \lambda^{N-1} \\
				&\quad \int_{B_1(\hat{x})} \frac{1}{\left|\xi^1\right|(1+\lambda\left|\xi^1\right|+\lambda|\xi^2-\hat{z}|)^2}\Bigl(\frac{(1+\lambda\left|\xi^1\right|+\lambda|\xi^2-\hat{z}|)^\tau}{\lambda^\tau} \frac{1}{(1+\lambda\left|\xi^1\right|+\lambda|\xi^2-\hat{z}|)^{\frac{N-2}{2}+\tau}}\Bigr)^2 \\
				& \leq C\left(\lambda^{-\tau}\|\omega\|_{*}\right)^2.
			\end{aligned}
		\end{equation}
		On the other hand, similar to the proof of \eqref{eqs2.29}, it follows that
		\begin{align}\label{eq3.21} 
			& \quad \int_{\Omega \backslash B_1(\hat{x})} \frac{(PW_{\hat{x},\lambda})^{2^{\sharp}-2}|\omega|^2}{\left|\xi^1\right|} \nonumber\\
			&\leq C\|\omega\|_{*}^2 \int_{\Omega \backslash B_1(\hat{x})} \frac{(PW_{\hat{x},\lambda})^{2^{\sharp}-2}}{\left|\xi^1\right|}\Bigl[\sum_{j=0}^{\infty} \frac{\lambda^{\frac{N-2}{2}}}{\left(1+\lambda\left|\xi^1\right|+\lambda\left|\xi^2-\hat{z}_L^{j}\right|\right)^{\frac{N-2}{2}+\tau}}\Bigr]^2 \nonumber\\
			& \leq C\|\omega\|_{*}^2 \int_{\Omega \backslash B_1(\hat{x})}\Bigl[\sum_{j=0}^{\infty} \frac{\lambda^{\frac{N-2}{2}}}{\left(1+\lambda\left|\xi^1\right|+\lambda\left|\xi^2-\hat{z}_L^{j}\right|\right)^{\frac{N-2}{2}+\tau}}\Bigr]^{2^{\sharp}} \\
			& \leq C\|\omega\|_{*}^2 \lambda^{N-1} \int_{\Omega \backslash B_1(\hat{x})}\Bigl(\frac{1}{(1+\lambda\left|\xi^1\right|+\lambda\left|\xi^2-\hat{z}\right|)^{N-2-\theta}} \nonumber\\
			&\quad +\frac{1}{(1+\lambda\left|\xi^1\right|+\lambda\left|\xi^2-\hat{z}\right|)^{\frac{N-2}{2}+\theta}} \frac{1}{(\lambda L)^{\frac{N-2}{2}-2 \theta}}\Bigr)^{2^{\sharp}} \nonumber\\
			& \leq C\|\omega\|_{*}^2\Bigl(\frac{1}{\lambda^{N-1-2^{\sharp} \theta}}+\frac{1}{(\lambda L)^{N-1-2^{\sharp}(2 \theta)}}\Bigr)\leq C \left(\lambda^{-\tau}\|\omega\|_*\right)^2. \nonumber
		\end{align}
		Combining \eqref{eq3.19}-\eqref{eq3.21} and Proposition \ref{pro3.1}, we obtain
		\begin{equation}\label{eq3.22}
			\begin{aligned}
				& \quad \Bigl|\int_{\Omega} \frac{M(\xi)\Bigl(\left(P W_{\hat{x},\lambda}+\omega\right)_{+}^{2^{\sharp}-1}-(PW_{\hat{x},\lambda})^{2^{\sharp}-1}-(2^{\sharp}-1)(PW_{\hat{x},\lambda})^{2^{\sharp}-2}\omega\Bigr) \partial_{i}\left(P W_{\hat{x},\lambda}\right)}{\left|\xi^1\right|}\Bigr|\\
				& \leq C\lambda^{\alpha(i)} \left(\lambda^{-\tau}\|\omega\|_*\right)^2 \leq \frac{C\lambda^{\alpha(i)}}{\lambda^N}.
			\end{aligned}
		\end{equation}
		From \eqref{eq2.18}, we have
		\begin{equation}\label{eq3.23}
			\Bigl|\int_{\Omega} \frac{M(\xi)\left(P W_{\hat{x},\lambda}\right)^{2^{\sharp}-2} \omega\partial_{i}\left(P W_{\hat{x},\lambda}\right)}{\left|\xi^1\right|}\Bigr| \leq \frac{C \lambda^{\alpha(i)}}{\lambda^{\beta}}\|\omega\|_{*}\leq \frac{C \lambda^{\alpha(i)}}{\lambda^{\beta+1+\theta}}. 
		\end{equation}
		
		Putting \eqref{eq3.18}, \eqref{eq3.22} and \eqref{eq3.23} into \eqref{eq3.17}, we obtain
		\begin{equation}\label{eq3.24}
			\begin{aligned}
				&\quad -\int_{\Omega} \Delta\left(P W_{\hat{x},\lambda}+\omega\right) \partial_{i}\left(P W_{\hat{x},\lambda}\right)-\int_{\Omega} \frac{M(\xi)\left(P W_{\hat{x},\lambda}+\omega\right)_{+}^{2^{\sharp}-1} \partial_{i}\left(P W_{\hat{x},\lambda}\right)}{\left|\xi^1\right|} \\
				& = -\int_{\Omega} \Delta (P W_{\hat{x},\lambda})\partial_{i}\left(P W_{\hat{x},\lambda}\right) -\int_{\Omega}\frac{M(\xi)\left(P W_{\hat{x},\lambda}\right)^{2^{\sharp}-1} \partial_{i}\left(P W_{\hat{x},\lambda}\right)}{\left|\xi^1\right|}+\lambda^{\alpha(i)}O\Bigl(\frac{1}{\lambda^N}+\frac{1}{\lambda^{\beta+1+\theta}}\Bigr).
			\end{aligned}
		\end{equation}
		
		Combining Propositions \ref{proA.3} and \ref{proA.4} with \eqref{eq3.24}, we see that \eqref{eq3.16} is equivalent to
		\begin{equation}\label{eq3.25}
			\begin{aligned}
				\frac{B_{i} \lambda \hat{z}_{i}}{\lambda^{\beta_{k+i}-1}}=O\Bigl(\frac{1}{\lambda^{\beta-1} L}+\frac{1}{\lambda^{\beta_{M}-2+\kappa}}+\frac{\left(\lambda \hat{z}_{i}\right)^2}{\lambda^{\beta_{k+i}-1}}+\frac{1}{\lambda^{N-1}}+\frac{1}{\lambda^{\beta+\theta}}\Bigr)
			\end{aligned}
		\end{equation}
		and 
		\begin{equation}\label{eq3.26}
			\begin{aligned}
				\frac{\left(2^{\sharp}-1\right) BD}{\lambda^{N-1} L^{N-2}} \sum_{j=1}^{\infty} \Gamma\left(P^{j}, 0\right)+\frac{F}{\lambda^{\beta+1}} \sum_{i \in J} a_{i} = o\Bigl(\frac{1}{\lambda^{\beta+1}}\Bigr),
			\end{aligned}
		\end{equation}
		where $B,D,F$ are positive constants and $B_{i}$ is a non-zero constant, $i=1,2,\cdots,N$. 
		
		Since $L \sim \lambda^{\frac{\beta}{N-2}-1}$, $\beta \leq \beta_{k+i} \leq \beta_{M}<\beta\left(1+\frac{1}{N-2}\right)$ and \eqref{aeq2.4}, we see that
		$$
		\frac{B_{i} \lambda \hat{z}_{i}}{\lambda^{\beta_{k+i}-1}} =\frac{1}{\lambda^{\beta-1} L} =\frac{\left(\lambda \hat{z}_{i}\right)^2}{\lambda^{\beta_{k+i}-1}} =O\Bigl(\frac{1}{\lambda^{\beta_{k+i}-1+\theta}}\Bigr).
		$$
		Clearly, \eqref{eq3.25} and \eqref{eq3.26} have a solution $\left(\hat{x}_{0,L}, \lambda_{L}\right)$, satisfying
		$$
		\left|\hat{x}_{0,L}\right| \leq \frac{C}{\lambda^{1+\theta}}, \ \lambda_{L}=L^{\frac{N-2}{\beta-N+2}}\left(C_{0}+o(1)\right),
		$$
		where
		$$
		C_{0}=\Bigl[\frac{-F\sum_{i \in J} a_{i} }{\left(2^{\sharp}-1\right) BD \sum_{j=1}^{\infty} \Gamma\left(P^{j}, 0\right)}\Bigr]^{\frac{1}{\beta-(N-2)}}>0.
		$$
		So we have proved that $u=P W_{\hat{x}_{0,L},\lambda_{L}}+w_{L}$ satisfies $-\Delta u(x)=M(x) \frac{u_{+}^{\frac{N}{N-2}}}{|y|}$ in $\Omega$. It follows from Lemma \ref{lmA.1} and Proposition \ref{pro3.1} that $u_{+} \not \equiv 0$.
		Notice that 
		\begin{equation*}
			u(x)=\int_{\Omega} G(x, \xi) M(\xi)\frac{u_{+}^{2^{\sharp}-1}(\xi)}{|\xi^1|} \mathrm{d}\xi >0.
		\end{equation*}
		Therefore, $u>0$ in $\Omega$. The proof of the Theorem \ref{th1.1} is finished.
	\end{proof}
	
	\appendix
	\section{Some Essential Estimates} \label{secA}
	In this section, we give some essential estimates which are independently interesting. We believe that they are useful to other related problems involving polyharmonic equations.
	
	\begin{lemma}\label{lmA.1}
		It holds that
		$$
		0\le P W_{\hat{x},\lambda}(x) \le \sum_{j=0}^{\infty} W_{\hat{x}_{L}^{j}, \lambda}(x).
		$$
	\end{lemma}
	\begin{proof}
		We have
		\begin{equation*}
			\begin{aligned}
				P W_{\hat{x},\lambda}(x) & =\int_{\Omega} G(x, \xi) \frac{W_{\hat{x},\lambda}^{2^{\sharp}-1}(\xi)}{|\xi^1|} \mathrm{d}\xi =\sum_{j=0}^{\infty} \int_{\Omega} \Gamma\left(x, \xi+P_{L}^{j}\right) \frac{W_{\hat{x},\lambda}^{2^{\sharp}-1}(\xi)}{|\xi^1|} \mathrm{d}\xi \\ 
				& \leq \sum_{j=0}^{\infty} \int_{\mathbb{R}^{N}} \Gamma\left(x, \xi+P_{L}^{j}\right) \frac{W_{\hat{x},\lambda}^{2^{\sharp}-1}(\xi)}{|\xi^1|} \mathrm{d}\xi =\sum_{j=0}^{\infty} W_{\hat{x},\lambda}\left(x+P_{L}^{j}\right)=\sum_{j=0}^{\infty} W_{\hat{x}_{L}^{j}, \lambda}(x).				
			\end{aligned}
		\end{equation*}
	\end{proof}
	
	Let $\varphi_{\hat{x},\lambda}=W_{\hat{x},\lambda}-P W_{\hat{x},\lambda}$. Then $\varphi_{\hat{x},\lambda}$ has the following expansion.
	\begin{lemma}\label{lmA.2}
		For $x \in B_1(0)$, it holds that
		\begin{equation}\label{eqA.1}
			\varphi_{\hat{x},\lambda}(x)=-\frac{B}{\lambda^{\frac{N-2}{2}}} \sum_{j=1}^{\infty} \Gamma\left(x, \hat{x}+P_{L}^{j}\right)+O\Bigl(\frac{1}{L^{N-2} \lambda^{\frac{N}{2}}}\Bigr),
		\end{equation}
		\begin{equation}\label{eqA.2}
			\frac{\partial \varphi_{\hat{x},\lambda}}{\partial \hat{z}_{i}}(x) =-\frac{B}{\lambda^{\frac{N-2}{2}}} \sum_{j=1}^{\infty} \frac{\partial \Gamma\left(x, \hat{x}+P_{L}^{j}\right)}{\partial \hat{z}_{i}}+O\Bigl(\frac{1}{L^{N-2} \lambda^{\frac{N}{2}}}\Bigr)
		\end{equation}
		and
		\begin{equation}\label{eqA.3}
			\frac{\partial \varphi_{\hat{x},\lambda}}{\partial \lambda}(x)=\frac{B(N-2)}{2 \lambda^{\frac{N}{2}}} \sum_{j=1}^{\infty} \Gamma\left(x, \hat{x}+P_{L}^{j}\right)+O\Bigl(\frac{1}{L^{N-2} \lambda^{\frac{N+2}{2}}}\Bigr),
		\end{equation}
	\end{lemma}
	where $B=\int_{\mathbb{R}^{N}} \frac{W_{0,1}^{2^{\sharp}-1}(\xi)}{|\xi^1|}\mathrm{d}\xi$.
	\begin{proof}
		Recall the definition $\hat{x}=\left(0,\hat{z}\right)$ and $\xi=\left(\xi^1,\xi^2\right)$.
		We have
		\begin{equation}\label{eqA.4}
			P W_{\hat{x},\lambda}(x)=\sum_{j=0}^{\infty} \int_{\Omega} \Gamma\left(x, \xi+P_{L}^{j}\right) \frac{W_{\hat{x},\lambda}^{2^{\sharp}-1}(\xi)}{|\xi^1|} \mathrm{d}\xi.
		\end{equation}								
		For $j=0$, it holds that			
		\begin{align*}				
			\int_{\Omega} \Gamma(x, \xi) \frac{W_{\hat{x},\lambda}^{2^{\sharp}-1}}{|\xi^1|} & =\int_{\mathbb{R}^{N}} \Gamma(x, \xi) \frac{W_{\hat{x},\lambda}^{2^{\sharp}-1}}{|\xi^1|}-\int_{\mathbb{R}^{N} \backslash \Omega} \Gamma(x, \xi) \frac{W_{\hat{x},\lambda}^{2^{\sharp}-1}}{|\xi^1|} \\
			& =W_{\hat{x},\lambda}(x)+O\Bigl(\int_{\mathbb{R}^{N} \backslash \Omega} \frac{1}{|x-\xi|^{N-2}} \frac{1}{|\xi^1| \left(|\xi^1|+|\xi^2-\hat{z}|\right)^{N} \lambda^{\frac{N}{2}}} \mathrm{d}\xi\Bigr)\\
			& =W_{\hat{x},\lambda}(x)+O\Bigl(\int_{\mathbb{R}^{N} \backslash \Omega} \frac{1}{|\xi^1| \left(|\xi^1|+|\xi^2-\hat{z}|\right)^{2N-2}} \frac{1}{\lambda^{\frac{N}{2}}} \mathrm{d}\xi\Bigr)\\
			& =W_{\hat{x},\lambda}(x)+O\Bigl(\frac{1}{L^{N-1}\lambda^{\frac{N}{2}}}\Bigr).
		\end{align*} 
		For $j \neq 0$, we have
		\begin{align*}
			& \quad \int_{\Omega} \Gamma\left(x, \xi+P_{L}^{j}\right) \frac{W_{\hat{x},\lambda}^{2^{\sharp}-1}(\xi)}{|\xi^1|} \\
			& =\int_{B_{\delta}(\hat{x})} \Gamma\left(x, \xi+P_{L}^{j}\right) \frac{W_{\hat{x},\lambda}^{2^{\sharp}-1}(\xi)}{|\xi^1|}+O\Bigl(\int_{\Omega \backslash B_{\delta}(\hat{x})} \Gamma\left(x, \xi+P_{L}^{j}\right) \frac{1}{|\xi^1|\left(|\xi^1|+|\xi^2-\hat{z}|\right)^{N}} \frac{1}{\lambda^{\frac{N}{2}}}\Bigr) \\
			& =\frac{1}{\lambda^{\frac{N-2}{2}}} \int_{B_{\delta \lambda}(0)} \Gamma\left(x,\lambda^{-1} \xi+\hat{x}+P_{L}^{j}\right) \frac{W_{0,1}^{2^{\sharp}-1}}{|\xi^1|} +O\Bigl(\frac{1}{|L \tilde{P}_{j}|^{N-2}} \frac{1}{\lambda^{\frac{N}{2}}}\Bigr) \\
			& =\frac{B \Gamma\left(x,\hat{x}+P_{L}^{j}\right)}{\lambda^{\frac{N-2}{2}}}+O\Bigl(\frac{1}{|L \tilde{P}_{j}|^{N-2}} \frac{1}{\lambda^{\frac{N}{2}}}\Bigr).
		\end{align*}
		So we have proved \eqref{eqA.1}. 
		
		Using \eqref{eqA.4}, we have
		\begin{equation*}\label{eqA.5}
			\frac{\partial P W_{\hat{x},\lambda}(x)}{\partial \hat{z}_{i}}=\left(2^{\sharp}-1\right) \sum_{j=0}^{\infty} \int_{\Omega} \Gamma\left(x, \xi+P_L^{j}\right) \frac{W_{\hat{x},\lambda}^{2^{\sharp}-2}}{|\xi^1|} \frac{\partial W_{\hat{x},\lambda}}{\partial \hat{z}_{i}}(\xi)\mathrm{d}\xi.
		\end{equation*}
		For $j=0$, it holds that
		\begin{equation*}\label{eqA.6}
			\begin{aligned}
				& \quad \left(2^{\sharp}-1\right) \int_{\Omega} \Gamma\left(x, \xi\right) \frac{W_{\hat{x},\lambda}^{2^{\sharp}-2}}{|\xi^1|} \frac{\partial W_{\hat{x},\lambda}}{\partial \hat{z}_{i}} \\
				& =\left(2^{\sharp}-1\right) \int_{\mathbb{R}^{N}} \Gamma\left(x, \xi\right) \frac{W_{\hat{x},\lambda}^{2^{\sharp}-2}}{|\xi^1|} \frac{\partial W_{\hat{x},\lambda}}{\partial \hat{z}_{i}} - \left(2^{\sharp}-1\right) \int_{\mathbb{R}^{N} \backslash \Omega} \Gamma\left(x, \xi\right) \frac{W_{\hat{x},\lambda}^{2^{\sharp}-2}}{|\xi^1|} \frac{\partial W_{\hat{x},\lambda}}{\partial \hat{z}_{i}} \\
				& =\int_{\mathbb{R}^{N}} \Gamma\left(x, \xi\right) (-\Delta)\Bigl(\frac{\partial W_{\hat{x},\lambda}}{\partial \hat{z}_{i}}\Bigr) + O\Bigl(\int_{\mathbb{R}^{N} \backslash \Omega} \frac{1}{|x-\xi|^{N-2}} \frac{1}{|\xi^1| \left(|\xi^1|+|\xi^2-\hat{z}|\right)^{N+1}} \frac{1}{\lambda^{\frac{N}{2}}}\Bigr)\\
				& =\frac{\partial W_{\hat{x},\lambda}}{\partial \hat{z}_{i}}(x) + O\Bigl(\frac{1}{L^{{N}}} \frac{1}{\lambda^{\frac{N}{2}}}\Bigr).
			\end{aligned}
		\end{equation*}
		For $j \neq 0$, we have
		\begin{align*}
			& \quad \left(2^{\sharp}-1\right) \int_{\Omega} \Gamma\left(x, \xi+P_L^{j}\right) \frac{W_{\hat{x},\lambda}^{2^{\sharp}-2}}{|\xi^1|} \frac{\partial W_{\hat{x},\lambda}}{\partial \hat{z}_{i}}=-\left(2^{\sharp}-1\right) \int_{\Omega} \Gamma\left(x, \xi+P_L^{j}\right) \frac{W_{\hat{x},\lambda}^{2^{\sharp}-2}}{|\xi^1|} \frac{\partial W_{\hat{x},\lambda}}{\partial \xi^{2}_i} \\
			& =-\left(2^{\sharp}-1\right) \int_{B_{\delta}(\hat{x})} \Gamma\left(x, \xi+P_L^{j}\right) \frac{W_{\hat{x},\lambda}^{2^{\sharp}-2}}{|\xi^1|} \frac{\partial W_{\hat{x},\lambda}}{\partial \xi^{2}_i}+O\Bigl(\int_{\Omega \backslash B_{\delta}(\hat{x})} \Gamma\left(x, \xi+P_L^{j}\right) \frac{W_{\hat{x},\lambda}^{2^{\sharp}-2}}{|\xi^1|} \frac{\partial W_{\hat{x},\lambda}}{\partial \xi^{2}_i}\Bigr) \\
			& = \int_{B_{\delta}(\hat{x})} \frac{\partial \Gamma\left(x, \xi+P_L^{j}\right)}{\partial \xi^{2}_i} \frac{W_{\hat{x},\lambda}^{2^{\sharp}-1}}{|\xi^1|}+O\Bigl(\frac{1}{|L \tilde{P}^{j}|^{N-2}} \frac{1}{\lambda^{\frac{N}{2}}}\Bigr)\\
			& =\frac{1}{\lambda^{\frac{N-2}{2}}} \int_{B_{\delta \lambda}(0)} \frac{\partial \Gamma\left(x, \lambda^{-1} \xi+\hat{x}+P_L^{j}\right)}{\partial \hat{z}_{i}} \frac{W_{0,1}^{2^{\sharp}-1}}{|\xi^1|}+O\Bigl(\frac{1}{\left|L P_{j}\right|^{N-2}} \frac{1}{\lambda^{\frac{N}{2}}}\Bigr) \\
			& =\frac{1}{\lambda^{\frac{N-2}{2}}} \frac{\partial \Gamma\left(x, \hat{x}+P_L^{j}\right)}{\partial \hat{z}_{i}} \int_{B_{\delta \lambda}(0)} \frac{W_{0,1}^{2^{\sharp}-1}}{|\xi^1|}+O\Bigl(\frac{1}{\left|L P_{j}\right|^{N-2}} \frac{1}{\lambda^{\frac{N}{2}}}\Bigr) \\
			& =\frac{B}{\lambda^{\frac{N-2}{2}}} \frac{\partial \Gamma\left(x, \hat{x}+P_L^{j}\right)}{\partial \hat{z}_{i}}+O\Bigl(\frac{1}{\left|L P_{j}\right|^{N-2}} \frac{1}{\lambda^{\frac{N}{2}}}\Bigr).
		\end{align*}
		Thus we have proved \eqref{eqA.2}.
		
		Using \eqref{eqA.4}, we also have
		\begin{equation*}\label{eqA.8}
			\frac{\partial P W_{\hat{x},\lambda}(x)}{\partial \lambda}=\left(2^{\sharp}-1\right) \sum_{j=0}^{\infty} \int_{\Omega} \Gamma\left(x, \xi+P_L^{j}\right) \frac{W_{\hat{x},\lambda}^{2^{\sharp}-2}}{|\xi^1|} \frac{\partial W_{\hat{x},\lambda}}{\partial \lambda}(\xi)\mathrm{d}\xi.
		\end{equation*}
		For $j=0$, it holds that
		\begin{equation*}
			\begin{aligned}
				& \quad \left(2^{\sharp}-1\right) \int_{\Omega} \Gamma\left(x, \xi\right) \frac{W_{\hat{x},\lambda}^{2^{\sharp}-2}}{|\xi^1|} \frac{\partial W_{\hat{x},\lambda}}{\partial \lambda} \\
				& =\left(2^{\sharp}-1\right) \int_{\mathbb{R}^{N}} \Gamma\left(x, \xi\right) \frac{W_{\hat{x},\lambda}^{2^{\sharp}-2}}{|\xi^1|} \frac{\partial W_{\hat{x},\lambda}}{\partial \lambda} - \left(2^{\sharp}-1\right) \int_{\mathbb{R}^{N} \backslash \Omega} \Gamma\left(x, \xi\right) \frac{W_{\hat{x},\lambda}^{2^{\sharp}-2}}{|\xi^1|} \frac{\partial W_{\hat{x},\lambda}}{\partial \lambda} \\
				& =\frac{\partial W_{\hat{x},\lambda}}{\partial \lambda}(x) + \frac{1}{\lambda}O\Bigl(\frac{1}{L^{N-2}} \frac{1}{\lambda^{\frac{N}{2}}}\Bigr).
			\end{aligned}
		\end{equation*}
		For $j \neq 0$, we have
		\begin{equation*}
			\begin{aligned}
				& \quad \left(2^{\sharp}-1\right) \int_{\Omega} \Gamma\left(x, \xi+P_L^{j}\right) \frac{W_{\hat{x},\lambda}^{2^{\sharp}-2}}{|\xi^1|} \frac{\partial W_{\hat{x},\lambda}}{\partial \lambda}=\int_{\Omega} \Gamma\left(x, \xi+P_L^{j}\right) \frac{\partial}{\partial \lambda}\Bigl(\frac{W_{\hat{x},\lambda}^{2^{\sharp}-1}}{|\xi^1|}\Bigr) \\
				& =\int_{B_{\delta}(x)} \Gamma\left(x, \xi+P_L^{j}\right) \frac{\partial}{\partial \lambda}\Bigl(\frac{W_{\hat{x},\lambda}^{2^{\sharp}-1}}{|\xi^1|}\Bigr)+O\Bigl(\int_{\Omega \backslash B_{\delta}(x)} \Gamma\left(x, \xi+P_L^{j}\right) \frac{\partial}{\partial \lambda}\Bigl(\frac{W_{\hat{x},\lambda}^{2^{\sharp}-1}}{|\xi^1|}\Bigr)\Bigr) \\
				& =\frac{\partial}{\partial \lambda}\Bigl(\frac{1}{\lambda^{\frac{N-2}{2}}} \int_{B_{\delta \lambda}(0)} \Gamma\left(x,\lambda^{-1} \xi+\hat{x}+P_{L}^{j}\right) \frac{W_{0,1}^{2^{\sharp}-1}}{|\xi^1|}\Bigr)+\frac{1}{\lambda} O\Bigl(\frac{1}{|L \tilde{P}_{j}|^{N-2}} \frac{1}{\lambda^{\frac{N}{2}}}\Bigr) \\
				& =\frac{\partial}{\partial \lambda}\Bigl(\frac{\Gamma\left(x, \hat{x}+P_{L}^{j}\right)}{\lambda^{\frac{N-2}{2}}}\Bigr) \int_{B_{\delta \lambda}(0)} \frac{W_{0,1}^{2^{\sharp}-1}}{|\xi^1|}+\frac{1}{\lambda} O\Bigl(\frac{1}{|L \tilde{P}_{j}|^{N-2}} \frac{1}{\lambda^{\frac{N}{2}}}\Bigr) \\
				& =-\frac{(N-2)B\Gamma\left(x, \hat{x}+P_{L}^{j}\right)}{2\lambda^{\frac{N}{2}}}+\frac{1}{\lambda} O\Bigl(\frac{1}{|L \tilde{P}_{j}|^{N-2}} \frac{1}{\lambda^{\frac{N}{2}}}\Bigr).
			\end{aligned}
		\end{equation*}
		Therefore, we have proved \eqref{eqA.3}.
	\end{proof}
	
	With Lemma \ref{lmA.2}, we are now prepared to estimate the following quantities, which will be used to determine $x_{0,L}$ and $\lambda_{L}$ for the bubbling solution of \eqref{eq1.1}:
	\begin{equation*}\label{eqA.9}
		-\int_{\Omega} \Delta (PW_{\hat{x},\lambda}) \partial_{i}\left(P W_{\hat{x},\lambda}\right) -\int_{\Omega}\frac{M(\xi)\left(P W_{\hat{x},\lambda}\right)^{2^{\sharp}-1} \partial_{i}\left(P W_{\hat{x},\lambda}\right)}{|\xi^1|}.
	\end{equation*}
	First, we have
	\begin{align*}
		&\quad -\int_{\Omega} \Delta (PW_{\hat{x},\lambda})\partial_{i}\left(P W_{\hat{x},\lambda}\right) -\int_{\Omega}\frac{M(\xi)\left(P W_{\hat{x},\lambda}\right)^{2^{\sharp}-1} \partial_{i}\left(P W_{\hat{x},\lambda}\right)}{|\xi^1|} \\
		& =\int_{\Omega} \frac{W_{\hat{x},\lambda}^{2^{\sharp}-1}}{|\xi^1|} \partial_{i}\left(P W_{\hat{x},\lambda}\right)-\int_{\Omega}\frac{M(\xi)\left(P W_{\hat{x},\lambda}\right)^{2^{\sharp}-1} \partial_{i}\left(P W_{\hat{x},\lambda}\right)}{|\xi^1|}.
	\end{align*}
	On the one hand, using Lemma \ref{lmA.1}, we have
	\begin{align*}
		&\quad \Bigl|\int_{\Omega \backslash B_1(\hat{x})} \frac{W_{\hat{x},\lambda}^{2^{\sharp}-1}}{|\xi^1|} \partial_{i}\left(P W_{\hat{x},\lambda}\right)-\int_{\Omega \backslash B_1(\hat{x})} \frac{M(x)\left(P W_{\hat{x},\lambda}\right)^{2^{\sharp}-1} \partial_{i}\left(P W_{\hat{x},\lambda}\right)}{|\xi^1|}\Bigr| \\
		& \leq C \lambda^{\alpha(i)}\Bigl(\int_{\Omega \backslash B_1(\hat{x})} \frac{W_{\hat{x},\lambda}^{2^{\sharp}-1}}{|\xi^1|} P W_{\hat{x},\lambda}+\int_{\Omega \backslash B_1(\hat{x})} \frac{\left(P W_{\hat{x},\lambda}\right)^{2^{\sharp}}}{|\xi^1|}\Bigr) \leq C \lambda^{\alpha(i)} \int_{\Omega \backslash B_1(\hat{x})} \frac{\Bigl(\sum_{j=0}^{\infty} W_{\hat{x}_L^{j}, \lambda}\Bigr)^{2^{\sharp}}}{|\xi^1|} \\
		& \leq C \lambda^{\alpha(i)+N-1} \\
		&\quad \times \int_{\Omega \backslash B_1(\hat{x})}\Bigl[\frac{1}{|\xi^1|(1+\lambda|\xi^1|+\lambda|\xi^2-\hat{z}|)^{N-2}}+\sum_{j=1}^{\infty} \frac{1}{|\xi^1|(1+\lambda|\xi^1|+\lambda|\xi^2-\hat{z}^j_L|)^{N-2}}\Bigr]^{2^{\sharp}} \\
		& \leq C \lambda^{\alpha(i)+N-1} \\
		&\quad \times \int_{\Omega \backslash B_1(\hat{x})}\Bigl[\frac{1}{|\xi^1|(1+\lambda|\xi^1|+\lambda|\xi^2-\hat{z}|)^{N-2}}+\frac{1}{|\xi^1|(1+\lambda|\xi^1|\lambda|\xi^2-\hat{z}|)^{\frac{N-2}{2}+\theta}} \frac{1}{(\lambda L)^{\frac{N-2 }{2}-\theta}} \Bigr]^{2^{\sharp}} \\
		& \leq C\lambda^{\alpha(i)-N+1}. 
	\end{align*}
	On the other hand, we write
	\begin{equation*}
		\begin{aligned}
			& \quad \int_{B_1(\hat{x})} \frac{W_{\hat{x},\lambda}^{2^{\sharp}-1}}{|\xi^1|} \partial_{i}\left(P W_{\hat{x},\lambda}\right)-\int_{B_1(\hat{x})}\frac{M(\xi)\left(P W_{\hat{x},\lambda}\right)^{2^{\sharp}-1} \partial_{i}\left(P W_{\hat{x},\lambda}\right)}{|\xi^1|} \\
			& =\int_{B_1(\hat{x})} \Bigl(\frac{W_{\hat{x},\lambda}^{2^{\sharp}-1}}{|\xi^1|}-\frac{\left(P W_{\hat{x},\lambda}\right)^{2^{\sharp}-1}}{|\xi^1|}\Bigr) \partial_{i}\left(PW_{\hat{x},\lambda}\right)-\int_{B_1(\hat{x})}\frac{(M(\xi)-1)\left(P W_{\hat{x},\lambda}\right)^{2^{\sharp}-1} \partial_{i}\left(P W_{\hat{x},\lambda}\right)}{|\xi^1|} \\
			& :=J_1-J_2.
		\end{aligned}
	\end{equation*}
	By Lemma \ref{lmA.2}, we have
	\begin{equation*}
		\begin{aligned}
			J_1 & =\int_{B_1(\hat{x})} \Bigl(\frac{W_{\hat{x},\lambda}^{2^{\sharp}-1}}{|\xi^1|}-\frac{\left(P W_{\hat{x},\lambda}\right)^{2^{\sharp}-1}}{|\xi^1|}\Bigr) \partial_{i}W_{\hat{x},\lambda}-\int_{B_1(\hat{x})} \Bigl(\frac{W_{\hat{x},\lambda}^{2^{\sharp}-1}}{|\xi^1|}-\frac{\left(P W_{\hat{x},\lambda}\right)^{2^{\sharp}-1}}{|\xi^1|}\Bigr) \partial_{i} \varphi_{\hat{x},\lambda} \\
			& =\left(2^{\sharp}-1\right) \int_{B_1(\hat{x})} \frac{W_{\hat{x},\lambda}^{2^{\sharp}-2} \varphi_{\hat{x},\lambda} \partial_{i}W_{\hat{x},\lambda}}{|\xi^1|}+\lambda^{\alpha(i)} O\Bigl(\int_{B_1(\hat{x})} \frac{W_{\hat{x},\lambda}^{2^{\sharp}-2} \varphi_{\hat{x},\lambda}^2}{|\xi^1|}\Bigr) \\
			& =\left(2^{\sharp}-1\right) \int_{B_1(\hat{x})} \frac{W_{\hat{x},\lambda}^{2^{\sharp}-2} \varphi_{\hat{x},\lambda} \partial_{i}W_{\hat{x},\lambda}}{|\xi^1|}+\lambda^{\alpha(i)} O\Bigl(\frac{1}{\lambda^{N-2}L^{2(N-2)}}\int_{B_1(\hat{x})} \frac{W_{\hat{x},\lambda}^{2^{\sharp}-2} }{|\xi^1|}\Bigr) \\
			& =\left(2^{\sharp}-1\right) \int_{B_1(\hat{x})} \frac{W_{\hat{x},\lambda}^{2^{\sharp}-2} \varphi_{\hat{x},\lambda} \partial_{i}W_{\hat{x},\lambda}}{|\xi^1|}+\lambda^{\alpha(i)} O\Bigl(\frac{1}{\lambda^{N-1}}\Bigr),
		\end{aligned}
	\end{equation*}
	where we use $N \geq 5$ in the last equality, and
	\begin{align*}
		J_2 & =\int_{B_1(\hat{x})}\frac{(M(\xi)-1)\left(P W_{\hat{x},\lambda}\right)^{2^{\sharp}-1} \partial_{i}W_{\hat{x},\lambda}}{|\xi^1|}-\int_{B_1(\hat{x})}\frac{(M(\xi)-1)\left(P W_{\hat{x},\lambda}\right)^{2^{\sharp}-1} \partial_{i} \varphi_{\hat{x},\lambda}}{|\xi^1|} \\
		& =\int_{B_1(\hat{x})}\frac{(M(\xi)-1)W_{\hat{x},\lambda}^{2^{\sharp}-1} \partial_{i}W_{\hat{x},\lambda}}{|\xi^1|}+\lambda^{\alpha(i)} O\Bigl(\int_{B_1(\hat{x})}\frac{|M(\xi)-1| W_{\hat{x},\lambda}^{2^{\sharp}-1} |\varphi_{\hat{x},\lambda}|}{|\xi^1|}\Bigr) \\ 
		& =\int_{B_1(\hat{x})}\frac{(M(\xi)-1)W_{\hat{x},\lambda}^{2^{\sharp}-1} \partial_{i}W_{\hat{x},\lambda}}{|\xi^1|}+\lambda^{\alpha(i)} O\Bigl(\frac{1}{\lambda^{\frac{N-2}{2}}} \int_{B_1(\hat{x})} \frac{|\xi|^{\beta} W_{\hat{x},\lambda}^{2^{\sharp}-1}}{|\xi^1|}\Bigr) \\
		& =\int_{B_1(\hat{x})}\frac{(M(\xi)-1)W_{\hat{x},\lambda}^{2^{\sharp}-1} \partial_{i}W_{\hat{x},\lambda}}{|\xi^1|}+\lambda^{\alpha(i)} O\Bigl(\frac{|\hat{x}|^{\beta}}{(\lambda L)^{N-2}}+\frac{1}{\lambda^{N}}\Bigr) \\
		& =\int_{B_1(\hat{x})}\frac{(M(\xi)-1)W_{\hat{x},\lambda}^{2^{\sharp}-1} \partial_{i}W_{\hat{x},\lambda}}{|\xi^1|}+\lambda^{\alpha(i)} O\Bigl(\frac{1}{\lambda^{2\beta}}+\frac{1}{\lambda^{N}}\Bigr), 
	\end{align*}
	since $|\hat{x}|=o\left(\frac{1}{\lambda}\right)$ in the assumption \eqref{aeq2.4}.
	Thus, we obtain
	\begin{equation}\label{eqA.11}
		\begin{aligned}
			& \quad -\int_{\Omega} \Delta (PW_{\hat{x},\lambda})\partial_{i}\left(P W_{\hat{x},\lambda}\right) -\int_{\Omega}\frac{M(\xi)\left(P W_{\hat{x},\lambda}\right)^{2^{\sharp}-1} \partial_{i}\left(P W_{\hat{x},\lambda}\right)}{|\xi^1|} \\
			& =\left(2^{\sharp}-1\right) \int_{B_1(\hat{x})} \frac{W_{\hat{x},\lambda}^{2^{\sharp}-2} \varphi_{\hat{x},\lambda} \partial_{i}W_{\hat{x},\lambda}}{|\xi^1|}-\int_{B_1(\hat{x})}\frac{(M(\xi)-1) W_{\hat{x},\lambda}^{2^{\sharp}-1} \partial_{i}W_{\hat{x},\lambda}}{|\xi^1|} \\
			& \quad+\lambda^{\alpha(i)} O\Bigl(\frac{1}{\lambda^{N-1}}\Bigr).
		\end{aligned}
	\end{equation}
	
	\begin{proposition}\label{proA.3}
		For $i=1, \cdots, h$, it holds that
		\begin{equation*}\label{eqA.12}
			\begin{aligned}
				& \quad -\int_{\Omega} \Delta (PW_{\hat{x},\lambda})\frac{\partial PW_{\hat{x},\lambda}}{\partial \hat{z}_{i}} -\int_{\Omega}\frac{M(\xi)\left(P W_{\hat{x},\lambda}\right)^{2^{\sharp}-1} }{|\xi^1|} \frac{\partial P W_{\hat{x},\lambda}}{\partial \hat{z}_{i}} \\
				& =\frac{B_{i} \lambda \hat{z}_{i}}{\lambda^{\beta_{k+i}-1}}+O\Bigl(\frac{1}{\lambda^{\beta-1} L}+\frac{1}{\lambda^{\beta_{M}-2+\kappa}}+\frac{\left(\lambda \hat{z}_{i}\right)^2}{\lambda^{\beta_{k+i}-1}}\Bigr),
			\end{aligned}				
		\end{equation*}
		where $B_{i}$ is a non-zero constant.
	\end{proposition}
	\begin{proof}
		On the one hand, we have
		\begin{equation}\label{eqA.13}
			\begin{aligned}
				& \quad \left(2^{\sharp}-1\right) \int_{B_1(\hat{x})} \frac{W_{\hat{x},\lambda}^{2^{\sharp}-2} \varphi_{\hat{x},\lambda} }{|\xi^1|} \frac{\partial W_{\hat{x},\lambda}}{\partial \hat{z}_{i}} \\
				& =-\frac{\left(2^{\sharp}-1\right) B}{\lambda^{\frac{N-2}{2}}} \sum_{j=1}^{\infty} \int_{B_1(\hat{x})} \frac{W_{\hat{x},\lambda}^{2^{\sharp}-2}}{|\xi^1|} \frac{\partial W_{\hat{x},\lambda}}{\partial \hat{z}_{i}} \Gamma\left(\xi, \hat{x}+P_L^{j}\right)+O\Bigl(\frac{\lambda}{L^{N-2} \lambda^{\frac{N}{2}}} \int_{B_1(\hat{x})} \frac{W_{\hat{x},\lambda}^{2^{\sharp}-1}}{|\xi^1|}\Bigr) \\
				& =\frac{B}{\lambda^{\frac{N-2}{2}}} \sum_{j=1}^{\infty} \int_{B_1(\hat{x})} \frac{\partial }{\partial \hat{z}_{i}}\left(W_{\hat{x},\lambda}^{2^{\sharp}-1}\right) \frac{\Gamma\left(\xi, \hat{x}+P_L^{j}\right)}{|\xi^1|}+O\Bigl(\frac{1}{(\lambda L)^{N-2}}\Bigr) \\
				& =-\frac{B}{\lambda^{\frac{N-2}{2}}} \sum_{j=1}^{\infty} \int_{B_1(\hat{x})} \frac{W_{\hat{x},\lambda}^{2^{\sharp}-1}}{|\xi^1|} \frac{\partial \Gamma\left(\xi, \hat{x}+P_L^{j}\right)}{\partial \hat{z}_{i}}+O\Bigl(\frac{1}{\lambda^{\beta}}\Bigr) \\
				& =-\left.\frac{B^2}{\lambda^{N-2} L^{N-1}} \sum_{j=1}^{\infty} \frac{\partial \Gamma\left(x, P_L^{j}\right)}{\partial \xi^2_{i}}\right|_{\xi=0}+O\Bigl(\frac{1}{(\lambda L)^{N-1}}+\frac{1}{\lambda^{\beta}}\Bigr)=O\Bigl(\frac{1}{\lambda^{\beta-1} L}\Bigr),
			\end{aligned}
		\end{equation} 
		where we use the assumption \eqref{aeq2.4}.
		
		On the other hand, we also have
		\begin{align}\label{eqA.14} 
			& \quad \int_{B_1(\hat{x})}\frac{(M(\xi)-1) W_{\hat{x},\lambda}^{2^{\sharp}-1} }{|\xi^1|} \frac{\partial W_{\hat{x},\lambda}}{\partial \hat{z}_{i}} \nonumber\\ 
			& =\int_{B_{\delta}(\hat{x})} \sum_{j=1}^{N} a_{j}\left|\xi_{j}\right|^{\beta_{j}} \frac{W_{\hat{x},\lambda}^{2^{\sharp}-1}}{|\xi^1|} \frac{\partial W_{\hat{x},\lambda}}{\partial \hat{z}_{i}}+O\Bigl(\frac{1}{\lambda^{\beta_{M}-2+\kappa}}\Bigr) \nonumber\\
			& =-\int_{B_{\delta \lambda}(0)} \sum_{j=1}^{N} a_{j}\left|\xi_{j}+\lambda \hat{x}_{j}\right|^{\beta_{j}} \frac{1}{\lambda^{\beta_{j}-1}} \frac{W_{0,1}^{2^{\sharp}-1}}{|\xi^1|} \frac{\partial W_{0,1}}{\partial \xi_{k+i}}+O\Bigl(\frac{1}{\lambda^{\beta_{M}-2+\kappa}}\Bigr) \\
			& =-\int_{B_{\delta \lambda}(0)} a_{k+i}\left|\xi_{k+i}+\lambda \hat{z}_{i}\right|^{\beta_{k+i}} \frac{1}{\lambda^{\beta_{k+i}-1}} \frac{W_{0,1}^{2^{\sharp}-1}}{|\xi^1|} \frac{\partial W_{0,1}}{\partial \xi_{k+i}}+O\Bigl(\frac{1}{\lambda^{\beta_{M}-2+\kappa}}\Bigr) \nonumber\\
			& =-\frac{a_{k+i} \lambda \hat{z}_{i}}{\lambda^{\beta_{k+i}-1}} \int_{B_{\delta \lambda}(0)} \beta_{k+i}\left|\xi_{k+i}\right|^{\beta_{k+i}-2} \xi_{k+i} \frac{W_{0,1}^{2^{\sharp}-1}}{|\xi^1|} \frac{\partial W_{0,1}}{\partial \xi_{k+i}}+O\Bigl(\frac{1}{\lambda^{\beta_{M}-2+\kappa}}+\frac{\left(\lambda \hat{z}_{i}\right)^2}{\lambda^{\beta_{k+i}-1}}\Bigr) \nonumber\\
			& =-\frac{B_{i} \lambda \hat{z}_{i}}{\lambda^{\beta_{k+i}-1}}+O\Bigl(\frac{1}{\lambda^{\beta_{M}-2+\kappa}}+\frac{\left(\lambda \hat{z}_{i}\right)^2}{\lambda^{\beta_{k+i}-1}}\Bigr), \nonumber
		\end{align}
		where $\xi=(\xi_1, \cdots, \xi_{k+h})$ and
		\begin{equation*}
			\begin{aligned}
				B_{i}&=a_{k+i}\beta_{k+i} \int_{\mathbb{R}^{N}} \left|\xi_{k+i}\right|^{\beta_{k+i}-2} \xi_{k+i} \frac{W_{0,1}^{2^{\sharp}-1}}{|\xi^1|} \frac{\partial W_{0,1}}{\partial \xi_{k+i}} \mathrm{d} \xi 
			\end{aligned}
		\end{equation*}
		is a non-zero constant.
		
		Combining \eqref{eqA.11}, \eqref{eqA.13} and \eqref{eqA.14}, we have proved Proposition \ref{proA.3}. 
	\end{proof}

	\begin{proposition}\label{proA.4}
		It holds that
		\begin{equation*}\label{eqA.15}
			\begin{aligned}
				& \quad -\int_{\Omega} \Delta (PW_{\hat{x},\lambda})\frac{\partial PW_{\hat{x},\lambda}}{\partial \lambda} -\int_{\Omega}\frac{M(\xi)\left(P W_{\hat{x},\lambda}\right)^{2^{\sharp}-1} }{|\xi^1|} \frac{\partial P W_{\hat{x},\lambda}}{\partial \lambda} \\
				& =\frac{\left(2^{\sharp}-1\right) B D}{\lambda^{N-1} L^{N-2}} \sum_{j=1}^{\infty} \Gamma\left(P^{j}, 0\right)+\frac{F}{\lambda^{\beta+1}} \sum_{i \in J} a_{i} +o\Bigl(\frac{1}{\lambda^{\beta+1}}\Bigr),
			\end{aligned}
		\end{equation*}
		where $J=\left\{j: \beta_{j}=\beta\right\}$ and $D,F$ are positive constants. Here, $\psi_{0}$ is defined in \eqref{eq2.6} and $B$ is defined in Lemma \ref{lmA.2}. 
	\end{proposition}
	
	\begin{proof}
		On the one hand, similar to the proof of \eqref{eqA.13}, we have
		\begin{align}\label{eqA.16} 
			& \quad \left(2^{\sharp}-1\right) \int_{B_1(\hat{x})} \frac{W_{\hat{x},\lambda}^{2^{\sharp}-2} \varphi_{\hat{x},\lambda} }{|\xi^1|} \frac{\partial W_{\hat{x},\lambda}}{\partial \lambda} \nonumber\\
			& =-\frac{\left(2^{\sharp}-1\right) B}{\lambda^{\frac{N-2}{2}}} \sum_{j=1}^{\infty} \int_{B_1(\hat{x})} \frac{W_{\hat{x},\lambda}^{2^{\sharp}-2}}{|\xi^1|} \frac{\partial W_{\hat{x},\lambda}}{\partial \lambda} \Gamma\left(\xi, \hat{x}+P_L^{j}\right)+O\Bigl(\frac{1}{L^{N-2} \lambda^{\frac{N}{2}} \lambda} \int_{B_1(\hat{x})} \frac{W_{\hat{x},\lambda}^{2^{\sharp}-1}}{|\xi^1|}\Bigr) \nonumber\\
			& =-\frac{\left(2^{\sharp}-1\right) B}{\lambda^{N-1}} \sum_{j=1}^{\infty} \int_{B_{\lambda}(0)} \frac{W_{0,1}^{2^{\sharp}-2}}{|\xi^1|} \psi_{0} \Gamma\left(\lambda^{-1} \xi, P_L^{j}\right)+O\Bigl(\frac{1}{\lambda^{N} L^{N-2}}\Bigr) \\
			& =-\frac{\left(2^{\sharp}-1\right) B}{\lambda^{N-1}} \sum_{j=1}^{\infty} \int_{B_{\lambda}(0)} \frac{W_{0,1}^{2^{\sharp}-2}}{|\xi^1|} \psi_{0} \Gamma\left(0, P_L^{j}\right)+O\Bigl(\frac{1}{\lambda^{\beta+2}}\Bigr) \nonumber\\
			& =\frac{\left(2^{\sharp}-1\right) BD}{\lambda^{N-1} L^{N-2}} \sum_{j=1}^{\infty} \Gamma\left(P^{j}, 0\right)+O\Bigl(\frac{1}{\lambda^{\beta+2}}\Bigr), \nonumber
		\end{align}
		where 
		\begin{equation*}
			\begin{aligned}
				D:=-\int_{\mathbb{R}^{N}} \frac{W_{0,1}^{2^{\sharp}-2}}{|\xi^1|} \psi_{0}=-\frac{1}{2^{\sharp}-1}\left.\frac{\partial }{\partial \lambda}\right|_{\lambda=1}\int_{\mathbb{R}^{N}} \frac{W_{0,\lambda}^{2^{\sharp}-1}}{|\xi^1|}= \frac{N-2}{2(2^{\sharp}-1)}\int_{\mathbb{R}^{N}} \frac{W_{0,1}^{2^{\sharp}-1}}{|\xi^1|} >0
			\end{aligned}
		\end{equation*}
		is a positive constant. 
		
		On the other hand, we have
		\begin{equation}\label{eqA.17}
			\begin{aligned}
				& \quad \int_{B_1(\hat{x})}\frac{(M(\xi)-1) W_{\hat{x},\lambda}^{2^{\sharp}-1} }{|\xi^1|} \frac{\partial W_{\hat{x},\lambda}}{\partial \lambda} \\ 
				& =\int_{B_{\delta}(\hat{x})} \sum_{j=1}^{N} a_{j}\left|\xi_{j}\right|^{\beta_{j}} \frac{W_{\hat{x},\lambda}^{2^{\sharp}-1}}{|\xi^1|} \frac{\partial W_{\hat{x},\lambda}}{\partial \lambda}+O\Bigl(\frac{1}{\lambda^{\beta_{M}+\kappa}}\Bigr) \\
				& =\int_{B_{\delta \lambda}(0)} \sum_{j=1}^{N} a_{j}\left|\xi_{j}+\lambda \hat{x}_{j}\right|^{\beta_{j}} \frac{1}{\lambda^{\beta_{j}+1}} \frac{W_{0,1}^{2^{\sharp}-1}}{|\xi^1|} \psi_{0}+O\Bigl(\frac{1}{\lambda^{\beta_{M}+\kappa}}\Bigr)\\
				& =\int_{B_{\delta \lambda}(0)} \sum_{j=1}^{N} a_{j}\left|\xi_{j}\right|^{\beta_{j}} \frac{1}{\lambda^{\beta_{j}+1}} \frac{W_{0,1}^{2^{\sharp}-1}}{|\xi^1|} \psi_{0}+o\Bigl(\frac{1}{\lambda^{\beta+1}}\Bigr) \\ 
				& =\int_{B_{\delta \lambda}(0)} \sum_{i \in J} a_{i}\left|\xi_{i}\right|^{\beta} \frac{1}{\lambda^{\beta+1}} \frac{W_{0,1}^{2^{\sharp}-1}}{|\xi^1|} \psi_{0}+o\Bigl(\frac{1}{\lambda^{\beta+1}}\Bigr) \\
				& =\frac{1}{\lambda^{\beta+1}} \sum_{i \in J} a_{i} \int_{\mathbb{R}^{N}}\left|\xi_{i}\right|^{\beta} \frac{W_{0,1}^{2^{\sharp}-1}}{|\xi^1|} \psi_{0}+o\Bigl(\frac{1}{\lambda^{\beta+1}}\Bigr).
			\end{aligned}
		\end{equation}
		Noting that there hold that 
		\begin{equation*}
			\begin{aligned}
				\int_{\mathbb{R}^{N}}\left|\xi_{i}\right|^{\beta} \frac{W_{0,1}^{2^{\sharp}-1}}{|\xi^1|} \psi_{0}=\frac{1}{2^{\sharp}}\left.\frac{\partial }{\partial \lambda}\right|_{\lambda=1}\int_{\mathbb{R}^{N}}\left|\xi_{i}\right|^{\beta} \frac{W_{0,\lambda}^{2^{\sharp}}}{|\xi^1|}=-\frac{\beta}{2^{\sharp}}\int_{\mathbb{R}^{N}}\left|\xi_{i}\right|^{\beta} \frac{W_{0,1}^{2^{\sharp}}}{|\xi^1|} <0
			\end{aligned}
		\end{equation*}
		and 
		\begin{equation*}
			\begin{aligned}
				\int_{\mathbb{R}^{N}}\left|\xi_{i}\right|^{\beta} \frac{W_{0,1}^{2^{\sharp}-1}}{|\xi^1|} \psi_{0}&=\int_{\mathbb{R}^{N}}\left|\xi_{j}\right|^{\beta} \frac{W_{0,1}^{2^{\sharp}-1}}{|\xi^1|} \psi_{0}, \ \text{for any} \ 1 \leq i,j \leq k \ \text{or} \ k+1 \leq i,j \leq N.
			\end{aligned}
		\end{equation*}
		Thus we get
		\begin{equation*}
			\begin{aligned}
				& \int_{B_1(\hat{x})}\frac{(M(\xi)-1) W_{\hat{x},\lambda}^{2^{\sharp}-1} }{|\xi^1|} \frac{\partial W_{\hat{x},\lambda}}{\partial \lambda}=-\frac{F}{\lambda^{\beta+1}} \sum_{i \in J} a_{i} +o\Bigl(\frac{1}{\lambda^{\beta+1}}\Bigr),
			\end{aligned}
		\end{equation*} 
		where 
		$$ 
		F:=-\int_{\mathbb{R}^{N}}\left|\xi_{i}\right|^{\beta} \frac{W_{0,1}^{2^{\sharp}-1}}{|\xi^1|} \psi_{0}>0, \ i=1 \ \text{or} \ k+1
		$$
		is a positive constant. 
		
		Combining \eqref{eqA.11}, \eqref{eqA.16} and \eqref{eqA.17}, we finish the proof of Proposition \ref{proA.4}.
	\end{proof}
	
	\subsection*{Acknowledgements}
	The authors would like to thank professor Chunhua Wang from Central China Normal University for the helpful discussion.

\end{document}